\documentclass{article}

\usepackage[english]{babel}

\usepackage[letterpaper,top=2.5cm,bottom=2.5cm,left=3.5cm,right=3.5cm,marginparwidth=1.75cm]{geometry}

\usepackage{tabularx} % extra features for tabular environment
\usepackage{amsmath}  % improve math presentation
\usepackage{graphicx} % takes care of graphic including machinery

\usepackage{amsmath,amscd,amsbsy,amssymb,latexsym,url,bm,amsthm,mathrsfs}
\usepackage{epsfig,graphicx,subfigure}
\usepackage{enumerate,balance,mathtools}
\usepackage{wrapfig}
\usepackage{mathrsfs,euscript}
\usepackage[usenames]{xcolor}
\usepackage[vlined,boxed,commentsnumbered,linesnumbered,ruled]{algorithm2e}
\usepackage{amsfonts}
\usepackage{amsthm}
\usepackage{fancyhdr}
\usepackage{picinpar}
\usepackage{listings}
\usepackage{layout}
\usepackage[all,cmtip]{xy}
\usepackage{indentfirst}
\usepackage{tikz-cd}
\usepackage{tikz}
\usepackage[all]{xy}
\usepackage{mathrsfs}
\usepackage[colorlinks=true, allcolors=blue]{hyperref}
\usepackage{setspace}

\newtheorem{theorem}{Theorem}[section]
\newtheorem{lemma}[theorem]{Lemma}
\newtheorem*{convention}{Convention}

\newtheorem{remark}[theorem]{Remark}
\newtheorem{example}[theorem]{Example}
\newtheorem{proposition}[theorem]{Proposition}
\newtheorem{corollary}[theorem]{Corollary}

\newtheorem{defn}[theorem]{Definition}

\newtheorem{prop}[theorem]{Proposition}
\newtheorem{Ques}[theorem]{Question}

\makeatletter
\newcommand\ackname{Acknowledgements}
\if@titlepage
  \newenvironment{acknowledgements}{%
\else
  
\fi
\makeatother

\newcommand{\mb}[1]{\mathbb{#1}}

\newcommand{\mtc}[1]{\mathcal{#1}}

\DeclareMathOperator{\Gr}{Gr}
\DeclareMathOperator{\Bl}{Bl}

\DeclareMathOperator{\Pic}{Pic}
\DeclareMathOperator{\Hom}{Hom}
\DeclareMathOperator{\Coh}{Coh}
\DeclareMathOperator{\Ext}{Ext}

\DeclareMathOperator{\rk}{rank}

\DeclareMathOperator{\che}{ch}

\DeclareMathOperator{\ext}{ext}

\DeclareMathOperator{\DLP}{DLP}
\DeclareMathOperator{\Aut}{Aut}

\DeclareMathOperator{\ev}{ev}
\DeclareMathOperator{\coev}{coev}

\DeclareMathOperator{\gr}{gr}
\DeclareMathOperator{\Quot}{Quot}

\title{Moduli Spaces of Sheaves on General Blow-ups of $\mb{P}^2$}\date{}
\author{Junyan Zhao}

\begin{document}
\maketitle

\begin{abstract}
Let $X$ be the blow-up of $\mb{P}^2$ along $m$ general points, and $A=H-\sum \varepsilon_iE_i$ be a generic polarization with $0<\varepsilon_i\ll1$. We classify the Chern characters which satisfy the weak Brill-Noether property, i.e. a general sheaf in $M_A({\bf{v}})$, the moduli space of slope stable sheaves with Chern character ${\bf{v}}$, has at most one non-zero cohomology. We further give a necessary and sufficient condition for the existence of stable sheaves. Our strategy is to specialize to the case when the $m$ points are collinear.

\end{abstract}

\tableofcontents

\section{Introduction}

Let $X$ be a complex smooth projective surface with an ample $\mb{R}$-divisor $A$, and $M_{X,A}({\bf{v}})$ be the moduli space of Gieseker $A$-semistable sheaves of character ${\bf{v}}$. Among all the fundamental problems about moduli spaces of sheaves, there are two extremely interesting ones:

\begin{enumerate}[(1)]
    \item compute the cohomology of a general element in an irreducible component of $M_{X,A}({\bf{v}})$, and
    \item classify the Chern characters ${\bf{v}}$ for which $M_{X,A}({\bf{v}})$ is non-empty.
\end{enumerate}

In this paper, we classify the non-special characters and stable characters on the blow-up of $\mb{P}^2$ along $m$ general points by specializing to the case of $m$ collinear points and applying results from deformation theory.

\subsection{Prioritary sheaves}

In contrast to (semi)stable sheaves, the families of prioritary sheaves are easier to construct. Let $D$ be an effective divisor on $X$. A torsion-free coherent sheaf $\mtc{E}$ on $X$ is called \emph{$D$-prioritary} if
$$\Ext^2(\mtc{E},\mtc{E}(-D))=0.$$
For a character ${\bf{v}}\in K(X)_{\mb{Q}}$, let $\mtc{P}_D({\bf{v}})\subseteq \Coh({\bf{v}})$ be the open substack of $D$-prioritary sheaves. If $X$ is some blow-up of $\mb{P}^2$ with an exceptional divisor $E$ and a fibre $F:=H-E$, where $H$ is the pullback of a hyperplane section on $\mb{P}^2$, then stack $\mtc{P}_{F}({\bf{v}})$ of $F$-prioritary sheaves is irreducible by a theorem of Walter \cite{Wal98}. Let $E_i$ be the exceptional divisors. It is natural to take the polarization $A$ to be $H-\sum \varepsilon_iE_i$ with $0<\varepsilon_i\ll1$. Then every $\mu_A$-semistable sheaf is automatically $H$-prioritary. Therefore, if $M_{X,A}({\bf{v}})$ is nonempty, then it is an open dense substack of $\mtc{P}_H({\bf{v}})$. We prove the following result.

\begin{theorem} \textup{(Proposition \ref{2})}
Let $X$ be the blow-up of $\mb{P}^2$ along $m$ collinear points. Let ${\bf{v}}=(r,c_1,\Delta)$ be a Chern character such that $r\geq 2$ and $\Delta\geq 0$. Then the stack $\mtc{P}_{X,F}({\bf{v}})$ is non-empty, and a general sheaf $\mtc{E}$ parameterized by $\mtc{P}_{X,F}({\bf{v}})$ admits a resolution of the form $$0\longrightarrow\mtc{O}_X(-2H+D)^{\alpha}\oplus\mtc{O}_X(-H+D)^{\beta}\stackrel{s}{\longrightarrow}\bigoplus_{i=1}^m\mtc{O}_X(-E_i+D)^{\gamma_i}\oplus\mtc{O}_X(D)^{\delta}\longrightarrow \mtc{E}\longrightarrow0,$$ or $$0\longrightarrow\mtc{O}_X(-2H+D)^{\alpha}\stackrel{s}{\longrightarrow}\mtc{O}_X(-H+D)^{\beta}\oplus\bigoplus_{i=1}^m\mtc{O}_X(-E_i+D)^{\gamma_i}\oplus\mtc{O}_X(D)^{\delta}\longrightarrow \mtc{E}\longrightarrow0$$ for some divisor $D$. In particular, the stack $M_{X,A}({\bf{v}})$ is unirational when it is non-empty.
\end{theorem}

\subsection{Higher rank Brill-Noether theory}

The higher rank Brill-Noether theory aims to classify non-special Chern characters (Definition \ref{1}) on a polarized surface $(X,A)$. The applications have been found in classifying globally generated Chern characters (\cite{CH16}), describing effective cones of moduli spaces (\cite{Hui16}\cite{CHW17}), and classifying Chern characters with non-empty
moduli spaces $M_{X,A}({\bf{v}})$ (\cite{CH21}).

The classification of non-special Chern characters was worked out for $\mb{P}^2$ in \cite{GH98}, and for Hirzebruch surfaces, including $\mb{P}^1\times\mb{P}^1$ and $\mb{F}_1$, in \cite{CH16}. For del Pezzo surfaces and arbitrary
blow-ups, partial results were obtained in \cite{CH20} under the condition that the Chern character ${\bf{v}}$ satisfies $\chi({\bf{v}})=0$. For del Pezzo surfaces of degree $\geq 4$, a classification of all non-special Chern characters is given in \cite{LZ19}. In this paper, we classify the Chern characters for the blow-ups of $\mb{P}^2$ along $m$ general points which satisfies the weak Brill-Noether property.

\begin{defn}\label{1}
We say that the moduli space $M_{X,A}({\bf{v}})$ (resp. the moduli stack $\mtc{P}_H({\bf{v}})$) satisfies the \textup{weak Brill-Noether property} (or is \textup{non-special}) if there exists a sheaf $\mtc{E}\in M_{X,A}({\bf{v}})$ (resp. $\mtc{E}\in \mtc{P}_H({\bf{v}})$) such that $H^i(X,\mtc{E})\neq0$ for at most one $i$. In this case, we also say that the character ${\bf{v}}$ satisfies the \textup{weak Brill-Noether property} (or is \textup{non-special}).
\end{defn} 

Let $X$ be the blow-up of $\mb{P}^2$ along $m$ distinct points $p_1,...,p_m$, and $E_1,...,E_m$ be the corresponding exceptional divisors. When $p_1,...,p_m$ are collinear, we have the following.

\begin{theorem} \textup{(Theorem \ref{3})}
Let $X$ be the blow-up of $\mb{P}^2$ along $m$ collinear points. Let ${\bf{v}}=(r,\nu,\Delta)$ be a Chern character such that $r({\bf{v}})\geq 2$, and $\Delta\geq 0$. Write $\nu=aH-\sum b_iE_i$, and define $\nu':=aH-\sum_{b_i>0}b_iE_i$. If $\nu$ satisfies that $(\nu.E_i)\geq-1$ and $(\nu'.L)\geq-1$, then a general sheaf parameterized by $\mtc{P}_{F_1,...,F_m,H}({\bf{v}})$ is non-special.

\end{theorem}

We will see in Section 9 that stable sheaves can deform to nearby surfaces. Applying the semicontinuity of the dimensions of cohomology groups, one obtains the following result on general blow-ups.

\begin{theorem} \textup{(Theorem \ref{4})}
Let ${\bf{v}}=(r,\nu,\Delta)$ be a character such that $\nu({\bf{v}})=\alpha H-\sum \beta_iE_i$ with $-1\leq \beta_i\leq 0$, $\alpha-\sum\beta_i\geq-1$, and $\Delta({\bf{v}})\geq0$. Let $A=H-\sum \varepsilon_iE_i$ be a polarization on the blow-up of $\mb{P}^2$ along $m$ general points with $0<\varepsilon_i\ll1$. Let $X_{h_0}$ be the blow-up of $\mb{P}^2$ along $m$ collinear points, and $X_{h}$ be the blow-up of $\mb{P}^2$ along $m$ general points. If $M_{\mtc{X}_{h_0},A}({\bf{v}})$ is non-empty, then ${\bf{v}}$ is a character on $\mtc{X}_h$ satisfying the weak Brill-Noether property.
\end{theorem}

\subsection{Exceptional sheaves}

Recall that an exceptional bundle is a simple vector bundle $\mtc{E}$ with $\Ext^i(\mtc{E},\mtc{E})=0$ for $i>0$. On $\mb{P}^2$, there is a beautiful description of the Chern characters of exceptional bundles \cite{LeP97}. When $X$ is a del Pezzo surface, it is known that every torsion-free exceptional sheaf is locally free, constructible (see Definition 6.9) and stable with respect to the anticanonical polarization. See \cite{KO95} for a thorough study of exceptional objects on del Pezzo surfaces.

On the blow-up of $\mb{P}^2$ along $m$ distinct points, we don't know whether there are non-constructible exceptional bundles. However, we still have the following result on the stability of constructible ones.

\begin{theorem} \textup{(Theorem \ref{5})}
Let $X$ be the blow-up of $\mb{P}^2$ along $m$ distinct collinear points, and $A=H-\sum \varepsilon_iE_i$ with $0<\varepsilon_i\ll 1$ a generic polarization. If $\mtc{E}$ is a constructible exceptional bundle, then it is $\mu_A$-stable.
\end{theorem}

\subsection{Existence of stable sheaves}

On $\mb{P}^2$, the existence of stable sheaves is controlled by the exceptional bundles. Dr\'{e}zet and Le Potier construct a function $\delta:\mb{R}\rightarrow\mb{R}$ whose graph in the $(\mu,\Delta)$-plane, which completely determines when $M_H({\bf{v}})$ is nonempty. If $(\mu({\bf{v}}),\delta({\bf{v}}))$ lies above the graph of $\delta$, then $M_H({\bf{v}})$ is nonempty. Otherwise, $M_H({\bf{v}})$ is empty or ${\bf{v}}$ is semi-exceptional. See \cite{DLP85} and \cite{LeP97} for the argument and \cite{CH21} for a figurative illustration. The classification of semistable characters on $\mb{P}^1\times\mb{P}^1$ is worked out in \cite{Rud94}. For Hirzebruch surfaces and generic polarizations, the classification of non-empty moduli spaces is worked out by \cite{CH21}. The existence theorems for del Pezzo surfaces of degree $\geq 3$ with the anti-canonical polarization is given by \cite{LZ19}. 

As a consequence of the classification on $\mb{P}^2$, we are able to construct a family of stable bundles on general blow-ups of $\mb{P}^2$ by analyzing the special blow-up along collinear points. We define the weak DL-condition for a Chern character ${\bf{v}}$ in Section 8. Roughly speaking, it means that for a constructible exceptional bundle $\mtc{E}$ whose slope is close to the slope of ${\bf{v}}$, the Euler characteristic $\chi({\bf{v}},\mtc{E})$ or $\chi(\mtc{E},{\bf{v}})$ has the expected sign.

\begin{theorem} \textup{(Theorem \ref{6})}
Let $X$ be the blow-up of $\mb{P}^2$ along $m$ general points, $A=H-\sum \varepsilon_iE_i$ be a polarization with $0<\varepsilon_i\ll1$, and ${\bf{v}}=(r,\nu,\Delta)$ be a character such that $\nu({\bf{v}})=\alpha H-\sum \beta_iE_i$ with $-1\leq \beta_i\leq 0$ and $\Delta({\bf{v}})\geq0$. Suppose that $\alpha\notin \mathfrak{C}$, where $\mathfrak{C}$ is the set of $H$-slopes on $\mb{P}^2$ of exceptional bundles. If ${\bf{v}}$ satisfies the weak DL-condition, then $M_{X,A}({\bf{v}})\neq \emptyset$.
\end{theorem}

\subsection{The organization of the paper}

In Section 2, we recall the preliminary facts needed in the rest of the paper. In Section 3, we study the basic properties of the blow-up of $\mb{P}^2$, especially the cohomology of line bundles.

Section 4 compiles some useful properties and constructions of special prioritary sheaves. We construct a family of prioritary sheaves and characterize Chern characters on the blow-up of $\mb{P}^2$ along distinct collinear points that satisfy the weak Brill-Noether property. In Section 5, we prove the stability of constructible exceptional bundles, which is crucial in the classification of stable Chern characters. 

In sections 6, we define a sharp Bogomolov function and determine its basic properties. We primarily concentrate on the characters on the blow-up of $\mb{P}^2$ along collinear points. In this case, we study the stability of sheaves with respect to generic polarization in detail. In Section 7, using deformation theory, we generalize our results to general blow-ups of $\mb{P}^2$. 

Finally, in Section 8, we compute its ample cone of the Hilbert scheme of points on the blow-up of $\mb{P}^2$ along distinct collinear points. 
~\\

\textbf{Acknowledgements} The author is greatly indebted to his advisor, Izzet Coskun, for suggesting the
problem and many stimulating conversations. The author would also like to thank Lawrence Ein, Benjamin Gould, Shizhuo Zhang, Yeqin Liu, Sixuan Lou for helpful discussions and suggestions.

\section{Preliminaries}

\begin{convention}\textup{
By a surface, we mean a connected smooth projective algebraic surface over $\mb{C}$. All sheaves are coherent unless specified. For a surface $X$ and coherent sheaves $\mtc{E}$ and $\mtc{F}$, we write $h^i(X,\mtc{E})=\dim H^i(X,\mtc{E})$, $\hom(\mtc{E},\mtc{F})=\dim \Hom(\mtc{E},\mtc{F})$, and $\ext^i(\mtc{E},\mtc{F})=\dim \Ext^i(\mtc{E},\mtc{F})$.}
\end{convention}

\subsection{Chern characters and Riemann-Roch on surfaces}

Let $\mtc{E}$ be a torsion-free sheaf on a polarized surface $(X,A)$. Let $K(X)_{\mb{Q}}$ be the Grothendieck group of $X$ with $\mb{Q}$-coefficients. The Chern character $\che(\mtc{E})=(\che_0(\mtc{E}),\che_1(\mtc{E}),\che_2(\mtc{E}))$ is given by $$\che_0(\mtc{E})=r(\mtc{E}),\quad \che_1(\mtc{E})=c_1(\mtc{E}),\quad \che_2(\mtc{E})=\frac{c_1^2-2c_2}{2}.$$ We define the \emph{total slope} $\nu$, the \emph{$A$-slope} $\mu_A$ and the \emph{discriminant} $\Delta$ by $$\nu(\mtc{E})=\frac{c_1(\mtc{E})}{r(\mtc{E})},\quad \mu_A(\mtc{E})=\frac{c_1(\mtc{E}).A}{r(\mtc{E})},\quad \Delta(\mtc{E})=\frac{\mu(\mtc{E})^2}{2}-\frac{\che_2(\mtc{E})}{r(\mtc{E})}.$$
These quantities depend only on the Chern character of $\mtc{E}$ (and the polarization) but not on the particular sheaf. Given a Chern character ${\bf{v}}$, we define the \emph{total slope} $\nu$, the \emph{$A$-slope} $\mu_A$ and the \emph{discriminant} $\Delta$ of ${\bf{v}}$ by the same formulae.

Then we have the following Riemann-Roch formula for torsion-free sheaves $\mtc{E}$ and $\mtc{F}$: $$\chi(\mtc{E},\mtc{F})=r(\mtc{E})r(\mtc{F})(P(\nu(\mtc{F})-\nu(\mtc{E}))-\Delta(\mtc{E})-\Delta(\mtc{F})),$$ where $$P(\nu):=\chi(\mtc{O}_X)+\frac{\nu.(\nu-K_X)}{2}$$ is the Hilbert polynomial of $\mtc{O}_X$. In particular, taking $\mtc{E}=\mtc{O}_X$, this reduces to the usual Riemann-Roch formula $$\chi(\mtc{F})=r(\mtc{F})(P(\nu(\mtc{F}))-\Delta(\mtc{F})).$$

\subsection{Stability}

We now introduce the basic notions and properties about stability conditions. For more details, see \cite{HL10} or \cite{LeP97}.

The sheaf $\mtc{E}$ is called \emph{$\mu_A$-semistable} if every proper subsheaf $0\neq\mtc{F}\subseteq\mtc{E}$ of smaller rank satisfies $$\mu_A(\mtc{F})\leq\mu_A(\mtc{E}).$$ If the inequality is strict for every such $\mtc{F}$, then $\mtc{E}$ is called \emph{$\mu_A$-stable}.

Now suppose that $A$ is an integral divisor. Define the \emph{Hilbert polynomial} $P_{\mtc{E}}(m)$ and the \emph{reduced Hilbert polynomial} $p_{\mtc{E}}(m)$ of $\mtc{E}$ by $$P_{\mtc{E}}(m)=\chi(\mtc{E}(m)),\quad p_{\mtc{E}}(m)=\frac{\chi(\mtc{E}(m))}{r(\mtc{E})}.$$ Then $\mtc{E}$ is \emph{$A$-semistable} (or \emph{Gieseker semistable}) if every proper subsheaf $0\neq\mtc{F}\subseteq\mtc{E}$ of smaller rank satisfies that $$p_{\mtc{F}}(m)\leq p_{\mtc{E}}(m)$$ for all $m\gg0$. If the inequality is strict for every such $\mtc{F}$, then $\mtc{E}$ is called \emph{$A$-stable} (or \emph{Gieseker $A$-stable}). It follows immediately from the Riemann-Roch formula that $$\mu_A\textup{-stable}\Rightarrow A\textup{-stable}\Rightarrow A\textup{-semistable}\Rightarrow\mu_A\textup{-semistable}.$$ If $\mtc{E}$ is $\mu_A$-semistable for some polarization $A$, then $\Delta(\mtc{E})\geq0$ by Bogomolov inequality.

Every torsion free sheaf $\mtc{E}$ admits a \emph{Harder-Narasimhan filtration} with respect to both $\mu_A$- and $A$-stability, that is there is a finite filtration $$0=\mtc{E}_0\subset\mtc{E}_1\subset\mtc{E}_2\subset\cdots\subset\mtc{E}_n=\mtc{E}$$ such that the quotients $$\mtc{G}_i=\mtc{E}_i/\mtc{E}_{i-1},$$ called \emph{Harder-Narasimhan factors}, are $\mu_A$- (respectively $A$-) semistable and $$\mu_A(\mtc{G}_i)>\mu_A(\mtc{G}_{i+1})\quad\quad \textup{(respectively, }p_{\mtc{G}_i}(m)>p_{\mtc{G}_{i+1}}(m)\textup{ for }m\gg0\textup{)}$$ for $1\leq i\leq n-1$. Moreover, the Harder-Narasimhan filtration is unique.

\subsection{Prioritary sheaves}

In this section, we recall some results by Walter \cite{Wal98}.

\begin{defn}
Let $X$ be a surface, and $D$ be an effective divisor. A torsion free sheaf $\mtc{E}$ is called \textup{$D$-{prioritary}} if $$\Ext^2(\mtc{E},\mtc{E}(-D))=0.$$ 
\end{defn}

\begin{lemma}\textup{(Lemma 3.1 \cite{CH21})}
Let $D_1$ and $D_2$ be two effective divisors on a surface $X$ such that $D_1-D_2$ is effective. If a sheaf $\mtc{E}$ is $D_1$-prioritary, then it is also $D_2$-prioritary
\end{lemma}

\begin{lemma}\textup{(Lemma 4 \cite{Wal98})}
Let $D$ be an effective Cartier divisor on a surface $X$. If $\mtc{E}$ is a $D$-prioritary torsion-free sheaf on $X$ of character ${\bf{v}}$, then the restriction map $\mtc{P}_{D}({\bf{v}})\rightarrow \Coh_D(i^{*}{\bf{v}})$, given by $\mtc{E}\mapsto\mtc{E}|_{D}$, is smooth (and therefore open) in a neighborhood of $[\mtc{E}]$, where $i:D\rightarrow X$ is the natural inclusion.

\end{lemma}

\begin{defn}
A vector bundle $\mtc{E}$ of rank $r$ on $\mb{P}^1$ is called \textup{balanced} if  $\mtc{E}\simeq\mtc{O}_{\mb{P}^1}(a)^{k}\oplus \mtc{O}_{\mb{P}^1}(a+1)^{r-k}$ for some $a\in\mb{Z}$ and $0\leq k<r$ 
\end{defn}

\begin{lemma}\textup{(Lemma 3 \cite{Wal98})}
Let $r\geq2$ and $0\leq d<r$ be integers. Let $\Coh_{\mb{P}^1}(r,-d)$ be the stack of coherent sheaves of rank $r$ and degree $-d$ on $\mb{P}^1$.
\begin{enumerate}[(i)]
    \item If $d>0$, then the sheaves not balanced form a closed substack of $\Coh_{\mb{P}^1}(r,-d)$ of codimension at least $2$.
    \item If $d=0$, then the sheaves not balanced form a closed substack of $\Coh_{\mb{P}^1}(r,0)$ of codimension $1$.
\end{enumerate}

\end{lemma}
 
\begin{corollary}
If $\mtc{E}$ is a general sheaf in $\mtc{P}_F({\bf{v}})$ with $F$ a smooth rational curve, then $\mtc{E}|_F$ is balanced along $F$.
\end{corollary}
 
\begin{lemma}\textup{(Lemma 6 \cite{Wal98})}
Let $f:X'\rightarrow X$ be the blow-up of a surface $X$ at a point $x\in X$, and $E$ be the exceptional divisor in $X'$. Suppose that $\mtc{E}$ is a coherent sheaf of rank $r$ on $X'$ such that $\mtc{E}|_{E}\simeq \mtc{O}_E^{r-d}\oplus\mtc{O}_E(-1)^d$ for some $d$. Then $f_{*}\mtc{E}$ is locally free in a neighborhood of $x$, and there are exact sequences
$$0\longrightarrow f^{*}f_{*}\mtc{E}\longrightarrow \mtc{E}\longrightarrow\mtc{O}_E(-1)^{d}\longrightarrow0,$$
$$0\longrightarrow \mtc{E}(-E) \longrightarrow f^{*}f_{*}\mtc{E} \longrightarrow\mtc{O}_E^{r-d}\longrightarrow0.$$
Moreover, for any divisor $D$ on $X$, we have $\Ext^2(\mtc{E},\mtc{E}(f^{*}D)))\simeq \Ext^2(f_{*}\mtc{E},f_{*}\mtc{E}(D)))$. In particular, $\mtc{E}$ is $f^{*}D$-prioritary if and only if $f_{*}\mtc{E}$ is $D$-prioritary.

\end{lemma}

\begin{lemma}\textup{(Proposition 2 \cite{Wal98})}
Let $\pi:X\rightarrow C$ be a birationally ruled surface and $F\in NS(X)$ the numerical class of a fiber of $\pi$. Suppose $r\geq2$, $c_1\in NS(X)$, and $c_2\in \mb{Z}$ are given. Then the stack $\mtc{P}_{X,F}(r,c_1,c_2)$ of $F$-prioritary sheaves on $X$ of rank $r$ and Chern classes $c_1$ and $c_2$ is smooth and irreducible.
\end{lemma}

\subsection{Exceptional bundles}

In this section, we recall some known results of exceptional bundles.

A coherent sheaf $\mtc{E}$ on $X$ is called \emph{exceptional} if $\Hom(\mtc{E},\mtc{E})=\mb{C}$ and $\Ext^i(\mtc{E},\mtc{E})=0$ for any $i>0$. The Mukai's lemma (\cite{Muk87} \cite{KO95}) implies that every torsion free exceptional sheaves are locally free. Thus we call torsion free exceptional sheaves \emph{exceptional bundles}. There do exist torsion exceptional sheaves. For example, the structure sheaves $\mtc{O}_{E_i}$ of exceptional divisors on the blow-up of $\mb{P}^2$ along general points are exceptional.

Let $A=H-\sum\varepsilon_iE_i$ be a polarization on $X$, i.e. $\varepsilon_i>0$ and $\sum\varepsilon_i<1$. We say that $A$ is \emph{generic} (or $\varepsilon$ is generic) if $(\varepsilon_1,...,\varepsilon_m)$ is a generic point in the region defined by $\varepsilon_i>0$ and $\sum\varepsilon_i<1$. Sometimes, we prefer to take $A=H-\varepsilon\sum E_i$, that is, $\varepsilon_1=\cdots=\varepsilon_m$, whence we mean $\varepsilon$ is a generic number in $(0,1/m)$ by saying $A$ is generic or $\varepsilon$ is generic.

The following lemma is proved in \cite{CH21} on Hirzebruch surfaces. For the reader's convenience, we give the proof on our surface here.

\begin{lemma} \textup{(Lemma 6.7 \cite{CH21})}
Let ${\bf{v}}\in K(X)_{\mb{Q}}$ be a potentially exceptional character of rank $r$ with $c_1({\bf{v}})=aH-\sum b_iE_i$.
\begin{enumerate}[(1)]
    \item The discriminant of ${\bf{v}}$ is $\Delta=\frac{1}{2}-\frac{1}{2r^2}$.
    \item The character ${\bf{v}}$ is primitive.
    \item If $\mtc{E}$ is an $\mu_A$-stable sheaf of discriminant $\Delta(\mtc{E})<1/2$, then $\mtc{E}$ is exceptional.
    \item If $\varepsilon$ is generic and $\mtc{E}$ is a $\mu_A$-semistable sheaf of character ${\bf{v}}$, then it is $\mu_A$-stable and exceptional.
    \item If $\varepsilon$ is generic and $\mtc{E}$ is a $A$-semistable sheaf of discriminant $\Delta(\mtc{E})<\frac{1}{2}$, then it is semiexceptional.
\end{enumerate}
\end{lemma}

\begin{proof}
\begin{enumerate}[(1)]
    \item Solving the Riemann-Roch formula $$1=\chi({\bf{v}},{\bf{v}})=r^2(1-2\Delta)$$ for $\Delta$ proves the statement.
    \item By the Riemann-Roch formula, \begin{equation}\nonumber
        \begin{split}
            \chi({\bf{v}})&=r\left(P\left(\frac{a}{r}H-\sum \frac{b_i}{r}E_i\right)-\frac{1}{2}+\frac{1}{2r^2}\right)\\
            &=\frac{1}{2r}\left((a+2r)(a+r)-\sum b_i(b_i+r)+r^2+1\right).
        \end{split}
    \end{equation}
    As $\chi({\bf{v}})$ is an integer, then $\gcd(r,a,b_1,...,b_m)=1$ and thus ${\bf{v}}$ is primitive.
    \item By the Riemann-Roch formula and stability, one has $\hom(\mtc{E},\mtc{E})=1$, $\ext^2(\mtc{E},\mtc{E})=0$, and $$\chi(\mtc{E},\mtc{E})=\frac{1}{r^2}\left(1-2\Delta\right)=1-\ext^1(\mtc{E},\mtc{E})>0.$$ Thus $\ext^1(\mtc{E},\mtc{E})=0$ and $\mtc{E}$ is exceptional. 
    \item Since $\varepsilon$ is generic and ${\bf{v}}$ is primitive, then $\mtc{E}$ has no subsheaf of smaller rank with the same $A$-slope. Hence $\mtc{E}$ is $\mu_A$-stable, and exceptional by (3). 
    \item Since $\varepsilon$ is generic, then the Jordan-H\"{o}lder factors $\gr_1,...,\gr_l$ of $\mtc{E}$ have the same total slope and discriminant. They are also exceptional bundles, by (1), so their Chern characters are primitive, hence have the same rank, and they are the same. Thus the factors are all isomorphic, and an easy induction using $\ext^1(\gr_1,\gr_1)=0$ shows that $\mtc{E}\simeq \gr_1^{\oplus l}$.
\end{enumerate}
\end{proof}

The simplest examples of exceptional bundles on blow-ups of $\mb{P}^2$ are line bundles. Now given an ordered pair of sheaves $(\mtc{E},\mtc{F})$, we form the evaluation and coevaluation maps $$\ev:\mtc{E}\otimes\Hom(\mtc{E},\mtc{F})\longrightarrow\mtc{F},\quad \coev:\mtc{E}\longrightarrow\mtc{F}\otimes\Hom(\mtc{E},\mtc{F})^{*},$$ each of which is associated to the identity element of the space $\Hom(\mtc{E},\mtc{F})\otimes \Hom(\mtc{E},\mtc{F})^{*}$. If the evaluation map is surjective, then we consider the kernel $$0\longrightarrow L_{\mtc{E}}\mtc{F}\longrightarrow\mtc{E}\otimes\Hom(\mtc{E},\mtc{F})\longrightarrow\mtc{F}\longrightarrow0;$$ if the coevaluation map is injective, then we consider the cokernel $$0\longrightarrow \mtc{E}\longrightarrow \mtc{F}\otimes\Hom(\mtc{E},\mtc{F})^{*}\longrightarrow R_{\mtc{F}}\mtc{E}\longrightarrow0.$$

\begin{defn}
The sheaf $L_{\mtc{E}}\mtc{F}$ is the \textup{left mutation} of $\mtc{F}$ across $E$, and the sheaf $R_{\mtc{F}}\mtc{E}$ is the \textup{right mutation} of $\mtc{E}$ across $\mtc{F}$.
\end{defn}

If $(\mtc{E},\mtc{F})$ is an ordered pair of exceptional bundles, then the left and right mutations are exceptional whenever they are defined. This gives us a way of producing exceptional bundles.

\begin{example}
The Euler sequence $$0\longrightarrow\mtc{O}_{\mb{P}^2}\longrightarrow\mtc{O}_{\mb{P}^2}(1)\otimes\Hom(\mtc{O}_{\mb{P}^2},\mtc{O}_{\mb{P}^2}(1))^{*}\longrightarrow T_{\mb{P}^2}\longrightarrow0$$ implies that $T_{\mb{P}^2}$ is exceptional, since it is the right mutation $R_{\mtc{O}(1)}\mtc{O}$.
\end{example}

Start with a strong exceptional collection $\sigma_0=(\mtc{E}_0,...,\mtc{E}_n)$ on a surface $X$. A \emph{transformation} of the exceptional collection $\sigma_0$ is defined as a transformation of a pair of neighboring objects in this collection. One extends $\sigma$ to an infinite periodic collection $(\mtc{E}_i)_{i\in\mb{Z}}$ by setting $\mtc{E}_{i+(n+1)k}=\mtc{E}_i\otimes(\omega_X^{*})^{\otimes k}$ for $i=0,1,...,n$. We can also do mutations in collections: if $(\mtc{E}_i,\mtc{E}_{i+1})$ has a surjective evaluation map (resp. injective coevaluation map), then one replaces $(\mtc{E}_i,\mtc{E}_{i+1})$ by $(L_{\mtc{E}_i}\mtc{E}_{i+1},\mtc{E}_i)$ (resp. $(\mtc{E}_{i+1},R_{\mtc{E}_i+1}\mtc{E}_i)$). When the operations are defined, we can iterate mutations. Write $L_j(\mtc{E}_i)_{i\in\mb{Z}}$ for the left mutation $L_{\mtc{E}_j}$.

On the surface $X$ obtained by blow up $\mb{P}^2$ along $m$ distinct points, one has a standard exceptional collection
$$\sigma_0=\left(\mtc{O}(-2H),\mtc{O}(-H),\mtc{O}(-E_1),\cdots,\mtc{O}(-E_m),\mtc{O}\right).$$

\begin{defn}
A bundle $\mtc{E}$ on $X$ is called \textup{constructible} if it can be obtained by a sequence of mutations from the standard helix $\sigma_0$.
\end{defn}

\begin{theorem}\textup{(\cite{KO95})}
All exceptional bundles and helixes on del Pezzo surfaces are constructible.
\end{theorem}

\section{Blow-ups of the projective plane}

In this section, we review some properties of blow-ups of $\mb{P}^2$. We refer the reader to \cite{Har77} and \cite{Lar04} for definitions and details of the proof.

Let $\Gamma=\{p_1,...,p_m\}$ be a set of $m$ distinct collinear points on $\mb{P}^2$. Let $\pi:X=\Bl_{\Gamma}\mb{P}^2\rightarrow \mb{P}^2$ be the blow-up of $\mb{P}^2$ along $\Gamma$ with exceptional divisor $E_1,...,E_m$. Let $H=\pi^{*}\mtc{O}(1)$ be pull-back of the line class on $\mb{P}^2$, and $L=H-\sum_{i=1}^m E_i$ be the proper transform of the line passing through the $m$ points. Then we have $$\Pic(X)=\mb{Z}H\oplus \mb{Z}E_1\oplus\cdots\oplus\mb{Z}E_m,$$ and intersection numbers  $$H^2=1,\quad E_i^2=-1,\quad L^2=1-m, \quad L.E_i=1, \quad L.H=1, \quad E_i.E_j=0$$ for any $i\neq j$. The canonical divisor of $X$ is $K_X=-3H+\sum_{i=1}^mE_i$. If $D=aH-\sum b_iE_i$, then the Riemann-Roch formula reads $$\chi(\mtc{O}_X(D))=\frac{(a+1)(a+2)}{2}-\sum_{i=1}^m\frac{b_i(b_i+1)}{2}.$$ Notice that $E_1,...,E_m,L$ are effective and that $H,H-E_1,...,H-E_m$ are nef. Since the cones generated by them are dual to each other, then we know that the nef cone of $X$ is generated by $$H,H-E_1,...,H-E_m,$$ and the effective cone is generated by $$E_1,...,E_m,H-E_1-\cdots-E_m.$$ Equivalently, a divisor $D=aH-\sum b_iE_i$ is nef if and only if $a\geq \sum b_i$ and $b_i\geq0$ for all $i$, and a divisor $D'=a'H-\sum b_i'E_i$ is effective if and only if $a'\geq0$ and $a'\geq b_i'$ for all $i$. 

\subsection{Cohomology of line bundles}

If we know the dimension of the global sections of a line bundle $\mtc{O}_X(D)$, then by Serre duality, one has $$h^2(X,\mtc{O}_X(D))=h^0(X,\mtc{O}_X(K_X-D)).$$ One can also compute $h^1(X,\mtc{O}_X(D))$ via Riemann-Roch formula: $$h^1(X,\mtc{O}_X(D))=h^0(X,\mtc{O}_X(D))+h^0(X,\mtc{O}_X(K_X-D))-\frac{D.(D-K_X)}{2}-1.$$ In this way, we know all the cohomology of a line bundle.

Now we compute the global sections of line bundles on $X$. Let $D=aH-\sum b_iE_i$ be a divisor with $a\geq \max\{b_1,0\}$. If any $b_i<0$, then chasing the exact sequence $$0\longrightarrow \mtc{O}_X(D-E_i)\longrightarrow \mtc{O}_X(D)\longrightarrow \mtc{O}_{E_i}(b_i)\longrightarrow 0,$$ one has $h^0(X,\mtc{O}_X(D))=h^0(X,\mtc{O}_X(D-E_i))$, so one can replace $D$ by $D-E_i$. Repeating this, we may assume that $b_i\geq0$ for any $i$. 

Assume first that $D$ is nef. Chasing the exact sequence $$0\longrightarrow \mtc{O}_X((n-1)H)\longrightarrow \mtc{O}_X(nH)\longrightarrow \mtc{O}_{H}(n)\longrightarrow 0,$$ and using the vanishing of $h^1(X,\mtc{O}_X(nH))$ for $n\geq0$, one obtains that $$h^0(X,\mtc{O}_X(aH))=\frac{(a+1)(a+2)}{2}.$$ Now consider the exact sequence $$0\longrightarrow \mtc{O}_X(aH-bE_1)\longrightarrow \mtc{O}_X(aH-(b-1)E_1)\longrightarrow \mtc{O}_{E_1}(b-1)\longrightarrow 0$$ for $1\leq b\leq b_1$. Since $a\geq b_1$, one always has the surjection $$H^0(X,aH-(b-1)E_1)\twoheadrightarrow H^0(E_1,\mtc{O}_{E_1}(b-1)).$$ As a consequence, we get $$h^0(X,\mtc{O}_X(aH-b_1E_1))=\frac{(a+1)(a+2)}{2}-\frac{b_1(b_1+1)}{2}.$$ Repeating this for $E_2,...,E_m$ and using the assumption that $a\geq \sum b_1$, we deduce that 
\begin{equation}\label{7}
    h^0(X,\mtc{O}_X(D))=\frac{(a+1)(a+2)}{2}-\sum_i^{m}\frac{b_i(b_i+1)}{2}.
\end{equation}

If $D$ is not nef, then the short exact sequence $$0\longrightarrow \mtc{O}_X(D-L)\longrightarrow \mtc{O}_X(D)\longrightarrow \mtc{O}_{L}\left(a-\sum b_i\right)\longrightarrow 0$$ together with $a<\sum b_i$ implies that $h^0(X,\mtc{O}_X(D))=h^0(X,\mtc{O}_X(D-L))$. Thus we can replace $D$ by $D-L=(a-1)-\sum(b_i-1)E_i$. If some coefficient $b_i$ becomes negative, one uses the previous reduction to place $b_i$ by $0$. Repeating this, we reduce the computation to the case when $D$ is nef.

Now we can prove the following useful properties.

\begin{lemma}
Let $X$ be the blow-up of $\mb{P}^2$ along $m$ collinear points, and $D=aH-\sum_{i=1}^mb_iE_i$ be a divisor on $X$. 
\begin{enumerate}[(a)]
    \item If $a>-3$, then $h^2(X,\mtc{O}_X(D))=0$. In particular, any effective divisor $D$ has $h^2(X,\mtc{O}_D)=0$.
    \item If $D$ is a nef divisor, then $\mathcal{O}_X(D)$ has no higher cohomology.
    \item Let $I$ be a subset of $\{1,2,...,m\}$ and $j$ be any index in $\{1,2,...,m\}$. If $D$ is of either of the form $$-2H+\sum_{i\in I} E_i,\quad -H+\sum_{i\in I} E_i,\quad aH-(a+1)E_j,\quad -E_j+\sum_{i\in I,i\neq j}E_i,$$ then $\mtc{O}_X(D)$ has no cohomology.
    \item If $D$ is of the form $$D=aH-\sum_{i=1}^mb_iE_i$$ with $a=\sum_{i=1}^mb_i-1$ and $b_i\geq 0$ for any $i$, then $\mtc{O}_X(D)$ has no higher cohomology.
    \item Assume that $H^i(X,\mtc{O}_X(D))=0$ for $i>0$. If $D.E_j\geq 0$ (resp. $D.H\geq -2$), then $H^i(X,\mtc{O}_X(D+E_j))=0$ (resp. $H^i(X,\mtc{O}_X(D+H))=0$) for $i>0$.
\end{enumerate}
\end{lemma}

\begin{proof}
\begin{enumerate}[(a)]
    \item This is because the coefficient of $H$ in the Serre dual of $D$ is negative.
    \item Using (\ref{7}), one has that $$h^0(X,\mtc{O}_X(D))=\frac{(a+1)(a+2)}{2}-\sum_i^{m}\frac{b_i(b_i+1)}{2}.$$ Since $h^2(X,\mtc{O}_X(D))=0$ by (a), then $h^1(X,\mtc{O}_X(D))=0$ follows from the  Riemann-Roch formula.
    \item This is also a combination of the Riemann-Roch formula, Serre duality, and the computation of global sections.
    \item Assume that $b_1\geq b_2\geq \cdots\geq b_m$. If $b_m=0$, then we can reduce to the $m-1$ case. Thus we may assume that $b_m>0$. One has $$\chi(\mtc{O}_X(D))-\chi(\mtc{O}_X(D-L))=(a+1)-\sum_{i=1}^mb_i=0,$$ and $$h^0(\mtc{O}_X(D))=h^0(\mtc{O}_X(D-L)).$$ Since $D-L$ is a nef divisor, thus we conclude by (b).
    \item If $D$ is a class such that $H^i(X,\mtc{O}_X(D))=0$ for all $i>0$ and $E_j.D\geq0$, then $H^i(X,\mtc{O}_XX(D+E_j))=0$ for all $i>0$. To see this, consider the exact sequence $$0\longrightarrow\mtc{O}_X(D)\longrightarrow\mtc{O}_X(D+E_j)\longrightarrow\mtc{O}_{\mb{P}^1}(-1)\longrightarrow0.$$ Since $\mtc{O}_X(D)$ and $\mtc{O}_{\mb{P}^1}(-1)$ have no higher  cohomology, then $\mtc{O}_X(D+E_j)$ has no higher cohomology. Similar sequences imply the other statements.
\end{enumerate}

\end{proof}

\section{Cohomology of General Sheaves in $\mtc{P}_{X,F}({\bf{v}})$}

\subsection{Elementary transformations}

In this section, we first recall the minimal discriminant property of type 2 elementary transformations, and then we give an explicit construction of an elementary transformation following the method in \cite{LZ19}.

Consider the projection from $E_i$ to a general line $H$, and let $F_i=L-E_i$ be the class of the fiber of this projection. We write $F$ to denote the class of one of the fibers $F_i$ if it does not matter which fiber we choose. For a given Chern character ${\bf{v}}=(r,c_1,\Delta)$, we can give a necessary and sufficient condition for the existence of $H$-prioritary sheaves with character $\bf{v}$. 

Write $c_1({\bf{v}})=aH-\sum b_iE_i$ with $$a=a'r+a'',\quad b_i=b_i'r+b_i'',\quad 0\leq a'',b_i''<r, \quad \forall i=1,...,m.$$ Let $D=(a'+2)H-\sum(b_i+1)E_i$ be a divisor. Notice that there exists an $F$-prioritary sheaf with Chern character ${\bf{v}}$ if and only if there exists one with ${\bf{v}'}=(r,c_1-rD,\Delta)$ since any twist of a prioritary sheaf is also prioritary.

Recall that a \emph{type $1$ elementary transformation} of $\mtc{E}$ along $\mtc{F}$ is the kernel of a surjective map $\mtc{E}\rightarrow \mtc{F}$, and a \emph{type $2$ elementary transformation} of $\mtc{E}$ along $\mtc{F}$ is an extension of $\mtc{F}$ by $\mtc{E}$.

For a clearer description of the existence of prioritary sheaves, see Figure 1 in \cite{CH21}.

\begin{prop}\textup{(Proposition 4.9 \cite{LZ19})}
Let $X$ be the blow-up of $\mb{P}^2$ along $m$ collinear points, and $\mtc{E}$ be a sheaf on $X$ of rank $r\geq 2$. Suppose that $\mtc{E}$ is a type 2 elementary transformation of $\mtc{O}_X(-2H)^{\oplus (r-a'')}\oplus\mtc{O}_X(-H)^{\oplus a''}$ along $\oplus_{i=1}^m\mtc{O}_{E_i}(-1)^{\oplus (r-b_i'')}$. Then $\mtc{E}$ is an $H$-prioritary sheaf, and for any $D\in \Pic(X)$, there are no $H$-prioritary sheaves of the same rank and total slope as $\mtc{E}\otimes\mtc{O}_X(D)$ with strictly smaller discriminant. 

\end{prop}

Now we construct type 2 elementary transformations of $\mtc{O}_X(-2H)^{\oplus (r-a'')}\oplus\mtc{O}_X(-H)^{\oplus a''}$ along $\oplus_{i=1}^m\mtc{O}_{E_i}(-1)^{\oplus (r-b_i'')}$. Roughly speaking, we distribute $H$ and $E_i$ evenly among the direct summands. Then we will see that such sheaves are $H$-prioritary.

Consider the $r$-tuple $$\mtc{S}=\left(\mtc{O}\left(-2H\right),...,\mtc{O}\left(-2H\right),\mtc{O}\left(-H\right),...,\mtc{O}\left(-H\right)\right)$$ where the number of $\mtc{O}(-H)$ is $a''$. 

\begin{enumerate}[1)]
    \item Start with $i=1$. Twist each coordinate by $\mtc{O}(E_1)$ starting from left to right in $\mtc{S}$ until reaching the $(r-b_1'')$-th coordinate.
    \item Let $\mtc{S}'$ be the new $r$-tuple obtained from the previous step. Reorder the coordinates of $\mtc{S}$ by decreasing $L$-slope. If two distinct line bundles $\mtc{O}(D_1)$ and $\mtc{O}(D_2)$ have the same $L$-slope, then $\mtc{O}(D_1)$ sits to the left of $\mtc{O}(D_2)$ if either
    \begin{enumerate}[(a)]
        \item $D_1.H<D_2.H$
        \item or $D_1.H=D_2.H$ and there exists a $j$ such that $D_1.E_i=D_2.E_i$ for all $i<j$ and $D_1.E_j>D_2.E_j$.
    \end{enumerate}
    \item Repeat steps 1) and 2) using $E_{i+1}$
\end{enumerate}

We call such a bundle $\mtc{E}$ a good bundle. By construction, there is a unique (up to isomorphism) good bundle $\mtc{E}$ such that $r(\mtc{E})=r$ and $c_1(\mtc{V})=c_1$. Also, notice that good bundles are type $2$ elementary transformations.

\begin{lemma}
Good bundles are $H$-prioritary.
\end{lemma}

\begin{proof}
Notice that for any two summands $\mtc{O}(D_1)$ and $\mtc{O}(D_2)$ in the good bundles, the coefficient of $H$ in $D_2-D_1$ is at least $-1$. Therefore we have that $$\Ext^2(\mtc{O}(D_1),\mtc{O}(D_2-H))\simeq H^2(X,\mtc{O}(D_2-D_1-H))=0,$$ so $\Ext^2(\mtc{E},\mtc{E}(-H))=0$, $\mtc{E}$ is $H$-prioritary.

\end{proof}

\subsection{Construction of a complete family of $H$-prioritary sheaves}

In this section, we will construct a complete family of $H$-prioritary, hence $F_i$-prioritary, sheaves on the blow-up of $\mb{P}^2$ along $m$ distinct collinear points, parameterized by a rational variety. This will imply that $\mtc{P}_H(\bf{v})$ and $\mtc{P}_{F_i}(\bf{v})$ are unirational. In particular, if $M_{X,A}({\bf{v}})$ is non-empty, then it is unirational.

\begin{proposition}\label{2}
Let $X$ be the blow-up of $\mb{P}^2$ along $m$ collinear points. Let ${\bf{v}}=(r,c_1,\Delta)$ be a Chern character such that $r\geq 2$ and $\Delta\geq 0$. Then the stack $\mtc{P}_{X,F}({\bf{v}})$ is non-empty, and a general sheaf $\mtc{E}$ parameterized by $\mtc{P}_{X,F}({\bf{v}})$ admits a resolution of the form $$0\longrightarrow\mtc{O}_X(-2H+D)^{\alpha}\oplus\mtc{O}_X(-H+D)^{\beta}\stackrel{s}{\longrightarrow}\bigoplus_{i=1}^m\mtc{O}_X(-E_i+D)^{\gamma_i}\oplus\mtc{O}_X(D)^{\delta}\longrightarrow \mtc{E}\longrightarrow0,$$ or $$0\longrightarrow\mtc{O}_X(-2H+D)^{\alpha}\stackrel{s}{\longrightarrow}\mtc{O}_X(-H+D)^{\beta}\oplus\bigoplus_{i=1}^m\mtc{O}_X(-E_i+D)^{\gamma_i}\oplus\mtc{O}_X(D)^{\delta}\longrightarrow \mtc{E}\longrightarrow0$$ for some divisor $D$. If the coefficient of $E_i$ in $c_1-rD$ is $b_i$, then the exponents are given by $$\alpha=-\chi({\bf{v}}(-D-H)),\quad \delta=\chi({\bf{v}}(-D)),\quad \gamma_i=b_i,\quad \beta=\left|r+\alpha-\delta-\sum \gamma_i\right|.$$ In particular, the stack $\mtc{P}_{X,F}({\bf{v}})$ is unirational.

\end{proposition}

\begin{proof}
We first show that we can choose $D$ such that $\alpha,\beta,\gamma,\delta$ given above are all non-negative. Write $D=cH-\sum d_iE_i$ and $c_1-rD=(a-c)H-\sum b_iE_i$. We first fix the coefficients $d_i$ by making $0\leq b_i<r$.  

Write $a=a'r+a''$ such that $0\leq a''< r$. By the Riemann-Roch formula, we have that $$\chi({\bf{v}}(-(a'+2)H+\sum d_iE_i))=r\left(\frac{a''-2r}{r}\left(\frac{a''-2r}{r}+3\right)+1-\sum_i\frac{b_i}{r}\left(\frac{b_i}{r}+1\right)-\Delta\right)<0.$$ Thus we can choose $c$ to be the largest integer such that $\chi({\bf{v}}(-D))\geq0$ but $\chi({\bf{v}}(-D-H))<0$. Setting $$\mtc{U}:=\mtc{O}_X(-2H+D)^{\alpha}\oplus\mtc{O}_X(-H+D)^{\beta},\quad \text{and} \quad \mtc{V}:=\bigoplus_{i=1}^m\mtc{O}_X(-E_i+D)^{\gamma_i}\oplus\mtc{O}_X(D)^{\delta},$$ one has that the sheaf $\mtc{H}om(\mtc{U},\mtc{V})$ is globally generated. Since $\rk\mtc{V}-\rk\mtc{U}=r\geq 2$, then a general cokernel $\mtc{E}=\mtc{E}_s$ is a vector bundle, whose Chern character is given by ${\bf{v}}(\mtc{E})={\bf{v}}$. The same argument applies to the exact sequence $$0\longrightarrow\mtc{O}_X(-2H)^{\alpha}\longrightarrow \mtc{O}_X(-H)^{\beta}\oplus\bigoplus_{i=1}^m\mtc{O}_X(-E_i)^{\gamma_i}\oplus\mtc{O}_X^{\delta}\longrightarrow \mtc{E}\longrightarrow0.$$ The rest of the statement is the contents of the next two lemmas, which are proved in \cite{LeP97} for $\mb{P}^2$ and in \cite{CH16} for blow-ups of $\mb{P}^2$ with $\chi({\bf{v}})=0$.

\end{proof}

\begin{lemma}
A general cokernel $\mtc{E}$ constructed as above is $H$-prioritary, and hence $F_i$-prioritary for all $i$. 

\end{lemma}

\begin{proof}
We may assume that $\mtc{E}$ is a vector bundle. To check that $\mtc{E}$ is $F$-prioritary, we need to show that $\Ext^2(\mtc{E},\mtc{E}(-H))=0$. Applying $\Ext(\mtc{E},\cdot)$ to the exact sequence $$0\longrightarrow \mtc{O}_X(-2H-H)^{\alpha}\oplus\mtc{O}_X(-H-H)^{\beta}{\longrightarrow}\bigoplus\mtc{O}_X(-E_i-H)^{\gamma_i}\oplus\mtc{O}_X(-H)^{\delta}\longrightarrow \mtc{E}(-H)\longrightarrow0,$$ we notice that it suffices to prove that $\Ext^2(\mtc{E},\mtc{O}_X(-E_i-H))=0$ and that $\Ext^2(\mtc{E},\mtc{O}_X(-H))=0$ for $i=1,...,m$. Now applying $\Ext(\cdot,\mtc{O}_X(-E_i-H))$ to the sequence $$0\longrightarrow \mtc{O}_X(-2H)^{\alpha}\oplus\mtc{O}_X(-H)^{\beta}{\longrightarrow}\bigoplus\mtc{O}_X(-E_i)^{\gamma_i}\oplus\mtc{O}_X^{\delta}\longrightarrow \mtc{E}\longrightarrow0,$$ one obtains $$\Ext^1(\mtc{O}_X(-2H)^{\alpha}\oplus\mtc{O}_X(-H)^{\beta},\mtc{O}_X(-E_i-H))\longrightarrow \Ext^2(\mtc{E},\mtc{O}_X(-E_i-H))$$ $$\longrightarrow \Ext^2(\bigoplus\mtc{O}_X(-E_i)^{\gamma_i}\oplus\mtc{O}_X^{\delta},\mtc{O}_X(-E_i-H)).$$ This gives $\Ext^2(\mtc{E},\mtc{O}_X(-E_i-H))=0$. Similarly, one can apply $\Ext(\cdot,\mtc{O}_X(-H))$ to obtain $$\Ext^1(\mtc{O}_X(-2H)^{\alpha}\oplus\mtc{O}_X(-H)^{\beta},\mtc{O}_X(-H))\longrightarrow \Ext^2(\mtc{E},\mtc{O}_X(-H))$$ $$\longrightarrow \Ext^2(\bigoplus\mtc{O}_X(-E_i)^{\gamma_i}\oplus\mtc{O}_X^{\delta},\mtc{O}_X(-H)),$$ yielding $\Ext^2(\mtc{E},\mtc{O}_X(-H))=0$. The same argument applies to the exact sequence $$0\longrightarrow\mtc{O}_X(-2H)^{\alpha}\longrightarrow \mtc{O}_X(-H)^{\beta}\oplus\bigoplus_{i=1}^m\mtc{O}_X(-E_i)^{\gamma_i}\oplus\mtc{O}_X^{\delta}\longrightarrow \mtc{E}\longrightarrow0.$$

\end{proof}

\begin{lemma}
Let $$\mtc{U}=\mtc{O}_X(-2H)^{\alpha}\oplus\mtc{O}_X(-H)^{\beta}\quad\textup{and}\quad \mtc{V}=\bigoplus\mtc{O}_X(-E_i)^{\gamma_i}\oplus\mtc{O}_X^{\delta}$$ or $$\mtc{U}=\mtc{O}_X(-2H)^{\alpha}\quad\textup{and}\quad \mtc{V}=\mtc{O}_X(-H)^{\beta}\oplus\bigoplus\mtc{O}_X(-E_i)^{\gamma_i}\oplus\mtc{O}_X^{\delta}$$ be as above. Then the open dense subset $S\subset\Hom\left(\mtc{U},\mtc{V}\right)$ parameterizing locally free $F$-prioritary sheaves is a complete family of $F$-prioritary sheaves.
\end{lemma}

\begin{proof}
We only prove for the first case, and the second is the same. We need to check that the Kodaira-Spencer map $$\kappa:T_{s}S=\Hom(U,V)\longrightarrow \Ext^1(\mtc{E},\mtc{E})$$ is surjective. As the map $\kappa$ factors as the composition of two maps $$\Hom(\mtc{U},\mtc{V})\stackrel{\phi}{\longrightarrow} \Hom(\mtc{U},\mtc{E}) \stackrel{\psi}{\longrightarrow}\Ext^1(\mtc{E},\mtc{E}),$$ where $\phi$ and $\psi$ are given by applying $\Ext(\mtc{U},\cdot)$ and $\Ext(\cdot,\mtc{E})$, respectively: $$\Hom(\mtc{U},\mtc{V})\stackrel{\phi}{\longrightarrow}\Hom(\mtc{U},\mtc{E})\longrightarrow\Ext^1(\mtc{U},\mtc{U}),\quad \quad \Hom(\mtc{U},\mtc{E})\stackrel{\psi}{\longrightarrow} \Ext^1(\mtc{E},\mtc{E})\longrightarrow\Ext^1(\mtc{V},\mtc{E}).$$ Notice that we have $$\Ext^1(\mtc{U},\mtc{U})=0,\quad \Ext^1(\mtc{V},\mtc{V})=0,\quad \textup{and}\quad \Ext^2(\mtc{V},\mtc{U})=0.$$ Applying $\Ext(\mtc{V},\cdot)$ to the sequence $$0\longrightarrow \mtc{O}_X(-2H)^{\alpha}\oplus\mtc{O}_X(-H)^{\beta}\stackrel{s}{\longrightarrow}\bigoplus\mtc{O}_X(-E_i)^{\gamma_i}\oplus\mtc{O}_X^{\delta}\longrightarrow \mtc{E}\longrightarrow0,$$ one obtains that $$\Ext^1(\mtc{V},\mtc{V})\longrightarrow \Ext^1(\mtc{V},\mtc{E})\longrightarrow \Ext^2(\mtc{V},\mtc{U}).$$ Thus we conclude that $\Ext^1(\mtc{V},\mtc{E})=0$ and consequently $\kappa$ is surjective.

\end{proof}

\subsection{Brill-Noether property}

In this section, we will give a sufficient condition for the character ${\bf{v}}$ to satisfy the weak Brill-Noether property.

By semicontinuity, if $\mtc{E}$ is any sheaf with at most one cohomology, then the cohomology also vanishes for the general sheaf in any component of $\mtc{P}_H({\bf{v}})$ that contains $\mtc{E}$. If moreover, the moduli space $M_A({\bf{v}})$ is non-empty, then the general sheaf in $M_A({\bf{v}})$ has at most one non-zero cohomology.

\begin{lemma}\textup{(\cite{CH16})} 
Let $\mtc{L}$ be a line bundle on a smooth surface $X$. Let $\mtc{V}$ be a torsion-free sheaf on $X$, and let $\mtc{V}'$ be a general elementary modification of $\mtc{V}$ at a general point $p\in X$, defined as the kernel of a general surjection $\phi:\mtc{V}\rightarrow\mtc{O}_p$:
$$0\longrightarrow\mtc{V}'\longrightarrow\mtc{V}\stackrel{\phi}{\longrightarrow}\mtc{O}_p\longrightarrow0.$$
\begin{enumerate}
    \item If $\mtc{V}$ is $\mtc{L}$-prioritary, then $\mtc{V}'$ is $\mtc{L}$-prioritary.
    \item The sheaves $\mtc{V}$ and $\mtc{V}'$ have the same rank and $c_1$, and $$\chi(\mtc{V}')=\chi(\mtc{V})-1,$$ $$\Delta(\mtc{V}')=\Delta(\mtc{V})+\frac{1}{r}.$$
    \item We have $H^2(X,\mtc{V})\simeq H^2(X,\mtc{V}')$.
    \item If at least one of $H^0(X,\mtc{V})$ or $H^1(X,\mtc{V})$ is zero, then at least one of $H^0(X,\mtc{V}')$ or $H^1(X,\mtc{V}')$ is zero. In particular, if $H^2(X,\mtc{V})=0$ and $\mtc{V}$ is non-special, then $H^2(X,\mtc{V}')=0$ and $\mtc{V}'$ is also non-special.
\end{enumerate}

\end{lemma}

\begin{theorem}\label{3}

Let $X$ be the blow-up of $\mb{P}^2$ along $m$ collinear points. Let ${\bf{v}}=(r,\nu,\Delta)$ be a Chern character such that $r({\bf{v}})\geq 2$, and $\Delta\geq 0$. Write $\nu=aH-\sum b_iE_i$, and define $\nu':=aH-\sum_{b_i>0}b_iE_i$. If $\nu$ satisfies that $(\nu.E_i)\geq-1$ and $(\nu'.L)\geq-1$, then $\mtc{P}_{F_1,...,F_m,H}({\bf{v}})$ is non-special.

\end{theorem}

\begin{proof}
If $\nu=\nu'$, then we can find good bundles $\mtc{E}=\oplus\mtc{O}(D_i)$ such that $r(\mtc{E})=r$ and $\nu(\mtc{E})=\nu$. Then $\mtc{E}$ has no higher cohomology by our computation of cohomology of line bundles. Thus we can find a non-special sheaf in $\mtc{P}_{H}(\bf{v})$. It then follows that a general sheaf parameterized by $\mtc{P}_{F_1,...,F_m,H}({\bf{v}})$ is non-special.

If $\nu\neq\nu'$, we may assume that $-1\leq (\nu.E_m)=-d_m/r<0$. Consider the map $\pi:X_m\rightarrow X_{m-1}$ contracting $E$. Let $\mtc{E}$ be a general sheaf in $\mtc{P}_{H,F_1,...,F_m}(v)$. Then in particular $\mtc{E}$ is locally free and balanced along $E$. Thus $\mtc{E}|_E\simeq \mtc{O}_E^{a}\oplus \mtc{O}_E(-1)^{r-a}$ for some $a$. Taking the push-forward of the resolution $$0\longrightarrow\mtc{O}_X(-2H+D)^{\alpha}\stackrel{s}{\longrightarrow}\mtc{O}_X(-H+D)^{\beta}\oplus\bigoplus_{i=1}^m\mtc{O}_X(-E_i+D)^{\gamma_i}\oplus\mtc{O}_X(D)^{\delta}\longrightarrow \mtc{E}\longrightarrow0$$ of $\mtc{E}$, one gets $$0\longrightarrow\mtc{O}_{X_{m-1}}(-2H+D')^{\alpha}\stackrel{s}{\longrightarrow}\mtc{O}_{X_{m-1}}(-H+D')^{\beta}\oplus\bigoplus_{i=1}^{m-1}\mtc{O}_{X_{m-1}}(-E_i+D')^{\gamma_i}\oplus\mtc{O}_{X_{m-1}}(D')^{\delta}$$$$\longrightarrow \pi_{*}\mtc{E}\longrightarrow R^1\pi_{*}\mtc{O}_X(-2H+D)^{\alpha}=0$$ because $(-2H+D)|_E\simeq \mtc{O}_E(-1)$, where $D'=\pi_{*}D$. In particular, the higher direct image $R^if_{*}\mtc{E}$ vanish for all $i>0$ so that the cohomology of $\mtc{E}$ is the cohomology of $\pi_{*}\mtc{E}$. Moreover, $\pi_{*}\mtc{E}$ is locally free and prioritary with respect to $H,F_1,...,F_{m-1}$ and admits a desired resolution. Notice that the rational map $\pi_{*}:\mtc{P}_{F}({\bf{v}})\dashrightarrow  \mtc{P}_{F}(\pi_{*}({\bf{v}}))$ defined by $\mtc{E}\mapsto \pi_{*}\mtc{E}$ is dominant: the tangent map is $$\Ext^1(\mtc{E},\mtc{E})\longrightarrow\Ext^1(\pi_{*}\mtc{E},\pi_{*}\mtc{E})\simeq \Ext^1(\pi^{*}\pi_{*}\mtc{E},\mtc{E}),$$ whose cokernel is $$\Ext^2(\mtc{O}_{E_m}(-1)^{d_m},\mtc{E})\simeq H^0(\mtc{E}^{*}\otimes K_X|_{E_m}(-1))=H^0(X,\mtc{O}_{E_m}(-1)^{r-d}\oplus\mtc{O}_{E_m}(-2)^{d})=0.$$ Thus we reduce to the case when $\nu=\nu'$.

\end{proof}

\begin{corollary}
Let $X$ be the blow-up of $\mb{P}^2$ along $m$ distinct collinear points, and ${\bf{v}}=(r,c_1,\Delta)$ be a Chern character on $X$ such that $r({\bf{v}})\geq 2$, $\Delta\geq 0$, and $c_1$ is nef. Then  $\mtc{P}_{F_1,...,F_m,H}({\bf{v}})$ is non-special.
\end{corollary}

\begin{remark}
\textup{On a smooth del Pezzo surface, we expect that ${\bf{v}}$ is non-special for ${\bf{v}}=(r,c_1,\Delta)$ such that $\Delta\geq 0$ and that $(\nu.C)\geq -1$ for any negative curve $C$. On our surface $X=X_m$, this already fails for line bundles. Consider for example $m=5$ and the line bundle $D=2H+E_1-E_2-...-E_5$, which satisfies that $(D.E_i)\geq -1$ and $(D.L)\geq -1$. We have $h^0(D)=h^0(H)=3$ and $\chi(D)=6-4=2$, so that $h^1(D)=1$.}
\end{remark}

\section{Stability of exceptional bundles}

In this section, we prove the stability of the constructible exceptional bundles. Although it is unknown whether all exceptional bundles on the blow-up of $\mb{P}^2$ along $m$ collinear points are constructible, we will see in the next two sections that the constructible ones are sufficient to give us a description of stable characters.

\begin{proposition}
Let $X$ be the blow-up of $\mb{P}^2$ along $m$ distinct points (not necessarily collinear). If $\mtc{E}$ is an exceptional bundle which is balanced on $E_i$ for any $i$, then $\pi_{*}\mtc{E}$ is semi-exceptional.
\end{proposition}

\begin{proof}
Let $\mtc{E}$ be an exceptional bundle balanced on every $E_i$. By twisting with some line bundle, we may assume that $\mtc{E}|_{E_i}=\mtc{O}_{E_i}^{r-d_i}\oplus\mtc{O}_{E_i}(-1)^{d_i}$. Let $\pi:X\rightarrow \mb{P}^2$ be the blow-down map. We first show that $R^i\pi_{*}\mtc{E}=0$ for any $i>0$. For each exceptional divisor $E=E_i$, by the theorem on formal functions, the cohomology vanishes if and only if $\varinjlim H^1(E_n,\mtc{E}|_{E_n})=0$, where $E_n$ is the closed subscheme of $X$ defined by $\mtc{O}(-nE)$. For each $n\geq 1$, there are exact sequences on $X$
$$0\longrightarrow\mtc{O}_E(n)\longrightarrow\mtc{O}_{E_{n+1}}\longrightarrow\mtc{O}_{E_n}\longrightarrow0$$
Tensoring with $\mtc{E}$ and induction on $n$ gives the result.

Now consider the short exact sequence $$0\rightarrow\pi^{*}\pi_{*}\mtc{E}\rightarrow\mtc{E}\rightarrow\bigoplus_{i=1}^m\mtc{O}_{E_i}(-1)^{\oplus d_i}\longrightarrow0.$$ Applying the functor $\Hom(\cdot,\mtc{E})$, one gets the long exact sequence $$\Ext^1(\mtc{E},\mtc{E})\rightarrow\Ext^1(\pi^{*}\pi_{*}\mtc{E},\mtc{E})\rightarrow\bigoplus_{i=1}^m\Ext^2(\mtc{O}_{E_i}(-1)^{d_i},\mtc{E})\rightarrow\Ext^2(\mtc{E},\mtc{E})\rightarrow\Ext^2(\pi^{*}\pi_{*}\mtc{E},\mtc{E})\rightarrow0.$$ We have that $$\Ext^i(\pi^{*}\pi_{*}\mtc{E},\mtc{E})=H^i(X,\pi^{*}(\pi_{*}\mtc{E})^{\vee}\otimes\mtc{E})=H^i(\mb{P}^2,\pi_{*}(\pi^{*}(\pi_{*}\mtc{E})^{\vee}\otimes\mtc{E}))=\Ext^i(\pi_{*}\mtc{E},\pi_{*}\mtc{E}),$$ and that $$\Ext^2(\mtc{O}_{E_i}(-1),\mtc{E})\simeq \Hom(\mtc{E},\mtc{O}_{E_i}(-2))^{*}=0$$ as $\mtc{E}_{E_i}\simeq\mtc{O}_{E_i}^{r-d_i}\oplus\mtc{O}_{E_i}(-1)^{d_i}$. Since $\mtc{E}$ is exceptional, then $\pi_{*}\mtc{E}$ is semi-exceptional by the long exact sequence. 

\end{proof}

\begin{theorem}\label{5}
Let $X$ be the blow-up of $\mb{P}^2$ along $m$ distinct collinear points, and $A=H-\sum \varepsilon_iE_i$ with $0<\varepsilon_i\ll 1$ a generic polarization. If $\mtc{E}$ is a constructible exceptional bundle, then it is $\mu_A$-stable.
\end{theorem}

\begin{proof}
Notice first that constructible exceptional bundles are balanced on every $E_i$. It suffices to show that $\mtc{E}$ is $\mu_A$-semistable. Suppose otherwise, one has a destabilizing subsheaf $\mtc{F}\subseteq \mtc{E}$. Write $\nu(\mtc{E})=\alpha H-\sum \beta_iE_i$ and $\nu(\mtc{F})=\alpha' H-\sum \beta_i'E_i$, then we have $$\alpha-\sum \varepsilon_i\beta_i<\alpha'-\sum \varepsilon_i\beta_i'.$$ It follows that $\alpha\leq \alpha'$ by taking $\varepsilon_i$ sufficiently small. However, since $\pi_{*}\mtc{E}$ is an exceptional bundle on $\mb{P}^2$, in particular $\mu_H$-stable, then we get a contradiction.
\end{proof}

\begin{remark}\textup{
For a general exceptional bundle $\mtc{E}$, consider the exact sequence} $$0\longrightarrow\mtc{E}(-H)\longrightarrow\mtc{E}\longrightarrow\mtc{E}|_H\longrightarrow0.$$ \textup{As $\ext^1(\mtc{E},\mtc{E})=\ext^2(\mtc{E},\mtc{E})=0$, then applying $\Ext(\mtc{E},\cdot)$ to the sequence, one gets} $$\Ext^2(\mtc{E},\mtc{E}(-H))\simeq \Ext^1(\mtc{E},\mtc{E}|_H).$$ \textup{Notice that there exists a line $l_0\in |H|$ such that $\mtc{E}|_{l_0}$ is balanced if and only if $\mtc{E}$ is $H$-prioritary:} $$\Ext^1(\mtc{E},\mtc{E}|_H)\simeq \Ext^1(\mtc{E}|_H,\mtc{E}|_H)=H^1(l_0,(\mtc{E}\otimes\mtc{E})|_{l_0}).$$

\end{remark}

\begin{example}
\textup{On our polarized surface $(X,A)$ with $\pi:X\rightarrow \mb{P}^2$ the blow-up of $\mb{P}^2$ along $m$ collinear points and $A=H-\varepsilon\sum E_i$ an ample class, consider the bundle $\mtc{E}:=\pi^{*}\mtc{T}_{\mb{P}^2}$. Notice that} $$\Ext^{*}(\mtc{E},\mtc{E})=\Ext^{*}(\mtc{O},\pi^{*}(\mtc{T}_{\mb{P}^2}\otimes\Omega^1_{\mb{P}^2}))\simeq H^{*}(X,\pi^{*}(\mtc{T}_{\mb{P}^2}\otimes\Omega^1_{\mb{P}^2}))$$$$=H^{*}(\mb{P}^2,\mtc{T}_{\mb{P}^2}\otimes\Omega^1_{\mb{P}^2})\simeq \Ext^{*}(\mtc{T}_{\mb{P}^2},\mtc{T}_{\mb{P}^2}),$$ \textup{implying that $\mtc{E}$ is exceptional since $\mtc{T}_{\mb{P}^2}$ is. We will show that $\mtc{E}$ is $\mu_A$-stable for any sufficiently small $\varepsilon>0$. Suppose we have a sub-line bundle $\mtc{O}_X(D)$ of $\mtc{E}$ with $D=aH-\sum b_iE_i$, we claim that $a\leq 1$. If not, consider the short exact sequence} $$0\longrightarrow\mtc{O}_{X}(-D)\longrightarrow\mtc{O}_{X}(H-D)\longrightarrow\mtc{O}_H(H-D)\simeq \mtc{O}_{\mb{P}^1}(1-a)\longrightarrow0.$$ \textup{Taking the induced long exact sequence on cohomology, we see that the map $$H^1(X,\mtc{O}_X(-D))\longrightarrow H^1(X,\mtc{O}_X(H-D))$$ given by multiplication of a non-zero element in $H^0(X,\mtc{O}_X(H))$ is injective. Now consider the pull-back of the Euler sequence on $\mb{P}^2$ twisted by $\mtc{O}(-D)$:} $$0\longrightarrow\mtc{O}_X(-D)\stackrel{(x_0,x_1,x_2)}{\longrightarrow}\mtc{O}_X(H-D)^{\oplus 3}\longrightarrow\mtc{E}(-D)\longrightarrow0.$$ \textup{As a consequence, the induced exact sequence on cohomology} $$0=H^0(X,\mtc{O}_X(H-D))^{\oplus 3}\longrightarrow H^0(\mtc{E}(-D))\longrightarrow H^1(X,\mtc{O}_X(-D))\longrightarrow H^1(X,\mtc{O}_X(H-D))^{\oplus 3}$$ \textup{implies that $H^0(X,\mtc{E}(-D))=0$, which is a contradiction.}

\textup{We can choose $\varepsilon_0$ sufficiently small such that} $$\mu_A(\mtc{O}_X(D))<3/2=\mu_A(\mtc{T}_{\mb{P}^2})$$ \textup{for all sub-line bundles $O_X(D)$ of $\mtc{E}$ with $a=0$ or $a=1$ as there are only finitely many of them. Now we claim that $\mtc{E}$ is $\mu_A$-stable. It suffices to check that there does not exist any possible destabilizing sub-line bundle $\mtc{O}_X(D)$ with $a<0$. Suppose that $\mtc{O}_X(D)$ is such a line bundle: $D=aH-\sum b_iE_i$ with $a<0$ and} $$\mu_A(\mtc{O}_X(D))=a-\varepsilon_0\sum b_i\geq \frac{3}{2}.$$ \textup{If $b_i>0$ for some $i$, then we may replace $D$ by $D+E_i$: notice that $\mu_A(\mtc{O}_X(D))<\mu_A(\mtc{O}_X(D+E_i))$ and that} $$H^0(\mtc{E}(-D))=H^0(\mtc{E}(-D-E_i))$$ \textup{due to the short exact sequence} $$0\longrightarrow\mtc{E}(-D-E_i)\longrightarrow\mtc{E}(-D)\longrightarrow\mtc{E}(-D)|_{E_i}\longrightarrow0$$ \textup{and $\mtc{E}(-D)|_{E_i}\simeq\mtc{O}_{\mb{P}^1}(-b_i)\oplus \mtc{O}_{\mb{P}^1}(-b_i)$. Now we reduce to the case when $b_i\leq 0$ for all $i$. Consider the short exact sequence} $$0\longrightarrow\mtc{E}(-D-L)\longrightarrow\mtc{E}(-D)\longrightarrow\mtc{E}(-D)|_{L}\longrightarrow0.$$ \textup{We see that $\mtc{E}|_{L}=\pi^{*}\mtc{T}_{\mb{P}^2}|_{L}\simeq \mtc{O}_L(1)\oplus\mtc{O}_L(2)$ by performing the pull-back of the Euler sequence on $\mb{P}^2$ restricted to $L$. It thus follows that} $$\mtc{E}(-D)|_{L}\simeq \mtc{O}_L(1-a+\sum b_i)\oplus\mtc{O}_L(2-a+\sum b_i).$$ \textup{Observe that $2-a+\sum b_i<0$: suppose not, then one gets} $$2+\sum b_i\geq a\geq \varepsilon_0\sum b_i+\frac{3}{2}, \quad (1-\varepsilon_0)(\sum b_i)\geq -\frac{1}{2},$$ \textup{which is impossible. Hence we arrive at the case when $a\geq0$ by performing these two kinds of reduction. This case follows from our choice of $\varepsilon_0$.}

\textup{In fact, $\mtc{E}$ is $\mu_A$-stable for any $\varepsilon$ such that $A$ is ample, i.e. $0<\varepsilon<1/m$. By our argument above, it suffices to check the case where $a=0$ and $a=1$. When $a=0$, notice that $\mtc{E}(-2E_i)$ has no sections because every non-zero tangent vector field on $\mb{P}^2$ vanishes at a point of multiplicity at most one. Thus the sub-line bundle with maximal $A$-slope is $\mtc{O}_X(\sum E_i)$, which satisfies that $\mu_A(\mtc{O}_X(\sum E_i))=m\varepsilon<1<3/2$, and hence does not destabilize $\mtc{E}$ either.}

\textup{When $a=1$, Notice that every section of $\mtc{T}_{\mb{P}^2}(-1)$ vanishes at one point, hence the divisor corresponding to any section of $\mtc{E}(-H)$ is either trivial or $E_i$. In particular, the sub-line bundle with maximal $A$-slope is $\mtc{O}_X(H+E_i)$, whose $A$-slope is $1+\varepsilon<3/2$, and hence does not destabilize $\mtc{E}$.}

\textup{However, in the case when $a=1$, if we take our polarization to be $A=H-(1-\delta)E_1-\sum_{i\geq2} \varepsilon E_i$ such that $0<\delta\ll0$, $0<\varepsilon\ll0$ and $1-\delta+(m-1)\varepsilon<1$, then $\mtc{O}(H+E_1)$ destabilizes $\mtc{E}$. This suggests that we cannot expect that the exceptional bundles, even the constructible ones, to be $\mu_A$-stable for any polarization $A=H-\sum \varepsilon_iE_i$.}

\end{example}

\section{Existence of Stable Sheaves}

Let $X$ be the blow-up of $\mb{P}^2$ along $m$ distinct collinear points. In this section, we first  computationally determine whether the moduli space $M_{X,A}({\bf{v}})$ is nonempty. Then for a generic polarization and an arbitrary character except for one special case, we will give an equivalent condition for the existence of stable sheaves with this character.

\subsection{Generic polarization and sharp Bogomolov inequalities}

In this section we introduce functions of the slope which provide sharp Bogomolov-type inequalities for various stabilities. We follow the treatment in \cite{CH21}.

For a $\mu_A$-stable exceptional bundle $\mtc{E}$, we define a function 
\begin{equation}\nonumber
\DLP_{A,\mtc{E}}(\nu)=
\begin{cases}
P(\nu-\nu(\mtc{E}))-\Delta(\mtc{E}), & \textup{if }\frac{1}{2}K_X.A\leq (\nu-\nu(\mtc{E})).A<0\\
P(\nu(\mtc{E})-\nu)-\Delta(\mtc{E}), & \textup{if }0<(\nu-\nu(\mtc{E})).A\leq-\frac{1}{2}K_X.A\\
\max\{P(\pm(\nu(\mtc{E})-\nu))-\Delta(\mtc{E})\}, & \textup{if }(\nu-\nu(\mtc{E})).A=0
\end{cases}
\end{equation}

\begin{defn}
Let $A=H-\sum \varepsilon_iE_i$ be a polarization and $\mathfrak{C}_A$ be the set of $\mu_A$-stable exceptional bundles on $X$. Define functions $$\DLP_A(\nu)=\sup_{\mtc{E}\in \mathfrak{C}_A \atop |(\nu-\nu(\mtc{E})).A|\leq -\frac{1}{2}K_X.A}\DLP_{A,\mtc{E}}(\nu),$$ and 
$$\DLP^{<r}_A(\nu)=\sup_{\mtc{E}\in \mathfrak{C}_A, r(\mtc{E})<r \atop |(\nu-\nu(\mtc{E})).A|\leq -\frac{1}{2}K_X.A}\DLP_{A,\mtc{E}}(\nu).$$
\end{defn}

\begin{proposition}
Let $\varepsilon$ be generic.
\begin{enumerate}[(i)]
    \item If $\mtc{E}$ is an $A$-semistable exceptional bundle on $X$ of rank $r$, then $$\Delta(\mtc{E})\geq \DLP^{<r}_{A}(\nu(\mtc{E})).$$ 
    \item If $\mtc{E}$ is an $A$-semistable non-semiexceptional bundle on $X$, then $$\Delta(\mtc{E})\geq \DLP_{A}(\nu(\mtc{E})).$$ 
\end{enumerate}
\end{proposition}

\begin{proof}
If $\mtc{E}$ is a $\mu_A$-semistable sheaf with $$\frac{1}{2}K_X.A\leq \mu_A(\mtc{E})-\mu_A(\mtc{V})<0,$$ when $\hom(\mtc{V},\mtc{E})=\hom(\mtc{E},\mtc{V}(K_X))$ by stability and duality. Therefore $\chi(\mtc{V},\mtc{E})\leq 0$ and $$\Delta(\mtc{E})\geq P(\nu(\mtc{E})-\nu(\mtc{V}))-\Delta(\mtc{V}).$$ Likewise, if $$0<\mu_A(\mtc{E})-\mu_A(\mtc{V})\leq -\frac{1}{2}K_X.A,$$ then the inequality $\chi(\mtc{E},\mtc{V})\leq 0$ provides a lower bound $$\Delta(\mtc{E})\geq P(\nu(\mtc{V})-\nu(\mtc{E}))-\Delta(\mtc{V}).$$

Notice that if $\varepsilon$ is generic, then $(\nu-\nu(\mtc{V})).A=0$ happen only when $\nu=\nu(\mtc{V})$. Suppose that $\mtc{E}$ is $A$-semistable of total slope $\nu(\mtc{E})=\nu(\mtc{V})$. If $\Delta(\mtc{E})=\Delta(\mtc{V})<1/2$, then $\mtc{E}$ is semistable. If $\Delta(\mtc{E})\neq\Delta(\mtc{V})$, then either $\hom(\mtc{E},\mtc{V})=0$ or $\hom(\mtc{V},\mtc{E})=0$ by $A$-semistability, and in either case Riemann-Roch implies $$\Delta(\mtc{E})\geq\DLP_{A,\mtc{V}}(\nu)=\frac{1}{2}+\frac{1}{2r(\mtc{V})^2}.$$ Thus if $A$ is generic, then $\Delta\geq \DLP_{H,\mtc{E}}(\nu)$ whenever there is an $A$-semistable sheaf of total slope $\nu$ and discriminant $\Delta$ satisfying $|(\nu-\nu(\mtc{V})).A|\leq -\frac{1}{2}K_X.A$.

\end{proof}

\begin{defn}
Let $A=H-\sum \varepsilon_iE_i$ be a polarization and ${\bf{v}}$ be a Chern character. Define $$\delta_{A}^{\mu-s}(\nu)=\inf\left\{\Delta\geq\frac{1}{2}:\textup{there is a }\mu_A\textup{-stable sheaf of total slope }\nu \textup{ and discriminant } \Delta\right\}.$$ We similarly define functions $\delta^{s}_A$, $\delta^{ss}_A$, $\delta^{\mu-ss}_A$.
\end{defn}

It is immediate that $$\delta_A^{\mu-ss}(\nu)\leq\delta_A^{ss}(\nu)\leq\delta_A^{s}(\nu)\leq\delta_A^{\mu-s}(\nu).$$ Now let us compare the various $\delta$-functions in the case where the polarization $A=H-\sum \varepsilon_iE_i$ is generic.

\begin{theorem} \textup{(Theorem 9.2 \cite{CH21})}
Let $\nu\in \Pic(X)_{\mb{Q}}$, and $A$ be a generic polarization. Then $$\delta_A^{ss}(\nu)=\delta_A^s(\nu)=\delta_A^{\mu-s}(\nu).$$ If moreover there is no $\mu_A$-stable exceptional bundle of total slope $\nu$, then these numbers also equal $\delta_A^{\mu-ss}(\nu)$.
\end{theorem}

The main result about existence of sheaves with discriminant above $\delta_A^{\mu-s}(\nu)$ is the following:

\begin{theorem} \textup{(Theorem 9.7 \cite{CH21})}
Let ${\bf{v}}=(r,\nu,\Delta)\in K(X)$ and $A$ be any polarization.
\begin{enumerate}
    \item If $\Delta>\delta_A^{\mu-s}(\nu)$, then there are $\mu_A$-stable sheaves of character ${\bf{v}}$.
    \item If there is a non-exceptional $\mu_A$-stable sheaf of character ${\bf{v}}$, then $\Delta\geq \delta_A^{\mu-s}(\nu)$.
    \item If there is a $\mu_A$-stable sheaf of slope $\nu$ and discriminant $\delta^{\mu-s}_A(\nu)>\frac{1}{2}$, then non-exceptional $\mu_A$-stable sheaves of character ${\bf{v}}$ exist if and only if $\Delta\geq \delta_A^{\mu-s}(\nu)$.
\end{enumerate}
\end{theorem}

\subsection{Harder-Narasimhan filtration}

Let $X=X_m$ be the blow-up of $\mb{P}^2$ along $m$ collinear points, and $A=H-\sum\varepsilon_i E_i$ be a polarization of $X$, where $\varepsilon_i>0$ is a rational number such that $\sum\varepsilon_i<1$.

Let $\mtc{E}_s$ be a complete family of torsion-free coherent sheaves on $(X,A)$, which is both $H$-prioritary and $F_i$-prioritary for all $i$, parameterized by a smooth algebraic variety $S$. Consider the $r(\varepsilon)A$-Harder-Narasimhan filtration of a general sheaf $\mtc{E}_s$, where $r(\varepsilon)$ is the smallest positive integer such that $r(\varepsilon)A$ is an integral divisor. Suppose this Harder-Narasimhan filtration has length $l$, and the $r(\varepsilon)A$-semistable quotients $\gr_{i,s}$ have corresponding $r(\varepsilon)A$-Hilbert polynomial $P_i$, reduced $r(\varepsilon)A$-Hilbert polynomial $p_1>\cdots>p_l$, and Chern characters ${\bf{gr}}_i=(r_i,\nu_i,\Delta_i)$.

The next lemma is useful in bounding the polarization in Section 6.3. We include the proof here.

\begin{lemma} \textup{(Lemma 5.1 \cite{CH21})}
A general sheaf $\mtc{E}_s$ in this family satisfies that $$0\leq \mu_{\max,A}(\mtc{E}_s)-\mu_{\min,A}(\mtc{E}_s)\leq1.$$ 
\end{lemma}

\begin{proof}
First suppose $C$ is a smooth rational curve, and then general $\mtc{E}_s|_C$ is a locally free sheaf. Recall that if $\mtc{E}_s/S$ is a complete family of $\mtc{O}(C)$-prioritary sheaves which are locally free along $C$, then the general $\mtc{E}_s$ has restriction $\mtc{E}_s|_C$ which is balanced so that
$$\mu_{\max}(\mtc{E}_s|C)-\mu_{\min}(\mtc{E}_s|C)\leq 1.$$
Observe that $\mu_{\max,\mtc{O}_C}(\mtc{E}_s)\leq \mu_{\max}(\mtc{E}_s|C)$. Indeed, suppose $\mtc{F}\subseteq \mtc{E}_s$ is a subsheaf. Then
$$\mu_{\mtc{O}_C}(\mtc{F})=\mu(\mtc{F}|_C)\leq \mu_{\max}(\mtc{E}_s|C).$$
Analogously we have $\mu_{\min,\mtc{O}_C}(\mtc{E}_s)\geq \mu_{\min}(\mtc{E}_s|C)$, and we conclude that
$$\mu_{\max,\mtc{O}(C)}(\mtc{E}_s)-\mu_{\min,\mtc{O}(C)}(\mtc{E}_s)\leq 1$$
holds for a general $s\in S$. (Even if $L$ is not ample, we write for example $\mu_{\max,L}(\mtc{E})$ for the maximum $L$-slope of a subsheaf of $\mtc{E}$, if it exists. For $L=\mtc{O}(C)$, the above restriction argument shows the maximum exists.)

Now observe that if $\mtc{E}_s/S$ is a complete family of $(F_1,...,F_m,H)$-prioritary sheaves, then 
\begin{equation}\nonumber
    \begin{split}
      mu_{\max,A}(\mtc{E}_s)-\mu_{\max,A}(\mtc{E}_s)&=\mu_A(\gr_{1,s})-\mu_A(\gr_{l,s})\\
      &=(\nu_1-\nu_l).((1-\sum\varepsilon_i)H+\sum \varepsilon_i F_i)\\
      &\leq (1-\sum\varepsilon_i)+\sum\varepsilon_i=1.  
    \end{split}
\end{equation}

\end{proof}

\begin{lemma} \textup{(Lemma 5.2 \cite{CH21})}
With the notation above, we have $\chi({\bf{gr}}_i,{\bf{gr}}_j)=0$ for all $i<j$.
\end{lemma}

The following theorem provides an algorithm to determine stable characters inductively. We will use it to determine a class of special characters that cannot be detected by the weak DL condition given in Definition \ref{8}. This is proved for Hirzebruch surfaces in \cite{CH21}. We repeat the argument here for the reader's convenience.

\begin{theorem} \textup{(Theorem 5.3 \cite{CH21})}
Suppose ${\bf{w}}_1,...,{\bf{w}}_k\in K(X)$ are characters of positive rank satisfying the following properties:
\begin{enumerate}
    \item ${\bf{w}}_1+\cdots+{\bf{w}}_k={\bf{v}}$.
    \item $q_1>\cdots>q_k$, where $q_i$ is the reduced $A$-Hilbert polynomial corresponding to ${\bf{w}}_i$.
    \item $\mu_A({\bf{w}}_1)-\mu_A({\bf{w}}_k)\leq 1$.
    \item $\chi({\bf{w}}_i,{\bf{w}}_j)=0$ for $i<j$.
    \item The moduli space $M_{A}({\bf{w}}_i)$ is nonempty for each $i$.
\end{enumerate}
Then $k=l$ and ${\bf{gr}}_i={\bf{w}}_i$ for each $i$.
\end{theorem}

\begin{proof}
Pick $A$-semistable sheaves $\mtc{W}_i\in M_{A}({\bf{w}}_i)$ for each $i$, and consider the sheaf $$\mtc{U}:=\bigoplus_i\mtc{W}_i$$ so that $\mtc{U}$ has character ${\bf{v}}$ and the Harder-Narasimhan filtration of $\mtc{U}$ has factors $\mtc{W}_1,...,\mtc{W}_k$. Then by assumption $$\mu_{\max,A}(\mtc{U})-\mu_{\min,A}(\mtc{U})=\mu_A(\mtc{W}_1)-\mu_A(\mtc{W}_k)\leq1,$$ so that $$\Ext^2(\mtc{W}_i,\mtc{W}_j(-F_n))\simeq \Hom(\mtc{W}_j,\mtc{W}_i(K_X+F_n))^{*}=0,$$ $\mtc{U}$ is both $H$-prioritary and $F_n$-prioritary for each $n$.

Now we can construct a complete family $\mtc{U}_t/\Sigma$ parameterized by a smooth, irreducible variety $\Sigma$ such that $\mtc{U}=\mtc{U}_{t_0}$ for some $t_0\in \Sigma$. Let $d\gg 0$ be sufficiently large and divisible, let $\chi=\chi(\mtc{O}_X(-dA),\mtc{U})$, and consider the universal family of quotients $\mtc{U}_t/\Sigma$ on $\Sigma=\Quot(\mtc{O}_X(-dA)^{\chi},\che(\mtc{U}))$ parameterizing quotients $$0\longrightarrow\mtc{K}_t\longrightarrow\mtc{O}_X(-dA)^{\chi}\longrightarrow\mtc{U}_t\longrightarrow0.$$
Let $t_0\in\Sigma$ be the point corresponding to the canonical evaluation $$\mtc{O}_X(-dA)\otimes\Hom(\mtc{O}_X(-dA),\mtc{U})\longrightarrow\mtc{U}.$$ Then the tangent space to $\Sigma$ at a point $t$ corresponding to the previous short exact sequence is $\Hom(\mtc{K}_t,\mtc{U}_t)$, and $\Sigma$ is smooth at $t$ if $\Ext^1(\mtc{K}_t,\mtc{U}_t)=0$. Applying $\Hom(\cdot,\mtc{U}_t)$ to the exact sequence, one gets $$\Hom(\mtc{K}_t,\mtc{U}_t)\longrightarrow\Ext^1(\mtc{U}_t,\mtc{U}_t)\longrightarrow\Ext^1(\mtc{O}_X(-dA)^{\chi},\mtc{U}_t)\longrightarrow\Ext^1(\mtc{K}_t,\mtc{U}_t)\longrightarrow\Ext^2(\mtc{U}_t,\mtc{U}_t).$$ By passing to the open subset parameterizing locally free sheaves if necessary, we have
$$\ext^2(\mtc{U}_t,\mtc{U}_t)=\hom(\mtc{U}_t,\mtc{U}_t(K_X))=0$$ by our assumptions on the slopes. Since $d\gg0$, we have $\Ext^1(\mtc{O}_X(-dA)^{\chi},\mtc{U}_t)=0$ by Serre vanishing and boundedness of the Quot scheme. Therefore $\Ext^1(\mtc{K}_t,\mtc{U}_t)=0$ and $\Sigma$ is smooth at $t$, including at $t=t_0$. Furthermore, the Kodaira-Spencer map at $t$ is the natural map
$$T_t\Sigma=\Hom(\mtc{K}_t,\mtc{U}_t)\longrightarrow\Ext^1(\mtc{U}_t,\mtc{U}_t),$$
so the universal family on $\Sigma$ is complete at $t$, including at $t=t_0$. We have thus constructed the required complete family $\mtc{U}_t/\Sigma$.

Let $Q_i$ be the $A$-Hilbert polynomial corresponding to ${\bf{w}}_i$. Then by the same computation as in the previous lemma, the Schatz stratum $S_A(Q_1,...,Q_k)\subset\Sigma$ is smooth at $t_0$ of codimension $0$. Therefore the stratum is dense in $\Sigma$, and the general sheaf $\mtc{U}_t$ has an $A$-Harder-Narasimhan filtration with quotients of character ${\bf{w}}_i$. Thus ${\bf{gr}}_i={\bf{w}}_i$ and $k=l$.

\end{proof}

\subsection{Classification of stable characters}

In this section, we will give an equivalent condition for the existence of $\mu_A$-stable bundles of character ${\bf{v}}$ for some polarization $A=H-\sum \varepsilon_iE_i$.

\begin{defn}
A torsion-free coherent sheaf $\mtc{E}$ (or Chern character) satisfies the \textup{strong Dr\'{e}zet-Le Potier condition} (abbr. as \textup{strong DL condition}) if 
\begin{enumerate}[(a)]
    \item for every $\mu_A$-stable sheaf $\mtc{F}$ satisfying $r(\mtc{F})<r(\mtc{E})$ and $$\mu_A(\mtc{E})\leq\mu_A(\mtc{F})\leq \mu_A(\mtc{E})-A.K_X,$$ we have $\chi(\mtc{F},\mtc{E})\leq 0$;
    \item for every $\mu_A$-stable sheaf $\mtc{F}$ satisfying $r(\mtc{F})<r(\mtc{E})$ and $$\mu_A(\mtc{E})+A.K_X\leq\mu_A(\mtc{F})\leq \mu_A(\mtc{E}),$$ we have $\chi(\mtc{E},\mtc{F})\leq 0$.
\end{enumerate}

\end{defn}

\begin{lemma}
Suppose $\mtc{E}$ is a non-exceptional $\mu_A$-stable sheaf with $r(\mtc{E})\geq 2$. Then $\Delta(\mtc{E})\geq 1/2$ and $\mtc{E}$ satisfies strong DL condition.
\end{lemma}

\begin{proof}

We have that $\ext^2(\mtc{E},\mtc{E})=0$ and $\hom(\mtc{E},\mtc{E})=1$ by stability. As $\mtc{E}$ is not exceptional, then $$r(\mtc{E})^2(1-2\Delta(\mtc{E}))=\chi(\mtc{E},\mtc{E})=1-\ext^1(\mtc{F},\mtc{F})\leq 0,$$ which gives $\Delta(\mtc{E})\geq 1/2$. Now if $\mtc{F}$ is a $\mu_A$-stable bundle such that $r(\mtc{F})<r(\mtc{E})$ and $$\mu_A(\mtc{E})\leq\mu_A(\mtc{F})\leq \mu_A(\mtc{E})-A.K_X,$$ then we have $\hom(\mtc{F},\mtc{E})=0$ and $\ext^2(\mtc{F},\mtc{E})=\hom(\mtc{E},\mtc{F}(K_X))=0$ by stability; which implies that $\chi(\mtc{F},\mtc{E})\leq0$. The other condition is similar.

\end{proof}

\begin{defn}\label{8}
A torsion-free coherent sheaf $\mtc{E}$ (or Chern character) satisfies the \textup{weak Dr\'{e}zet-Le Potier condition} (abbr. as \textup{weak DL condition}) if 
\begin{enumerate}[(a)]
    \item for every exceptional bundle $\mtc{F}$ which is constructible and satisfies $r(\mtc{F})<r(\mtc{E})$ and $$\mu_A(\mtc{E})\leq\mu_A(\mtc{F})\leq \mu_A(\mtc{E})-A.K_X,$$ for some polarization $A$, we have $\chi(\mtc{F},\mtc{E})\leq 0$;
    \item for every exceptional bundle $\mtc{F}$ which is constructible and satisfies $r(\mtc{F})<r(\mtc{E})$ and $$\mu_A(\mtc{E})\leq\mu_A(\mtc{F})\leq \mu_A(\mtc{E})-A.K_X,$$ for some polarization $A$, we have $\chi(\mtc{E},\mtc{F})\leq 0$.
\end{enumerate}

\end{defn}

\begin{remark}
\textup{If $\mtc{E}$ is a $\mu_A$-stable sheaf for some polarization $A$, then $\mtc{E}$ satisfies both strong and weak DL conditions thanks to the stability condition and Serre duality.}
\end{remark}

Let ${\bf{v}}=(r,\nu,\Delta)$ be a character such that $\nu({\bf{v}})=\alpha H-\sum \beta_iE_i$ with $-1\leq \beta_i\leq 0$ and $\Delta({\bf{v}})\geq0$.  Suppose that $\alpha\notin \mathfrak{C}$, where $\mathfrak{C}$ is the set of $H$-slopes on $\mb{P}^2$ of exceptional bundles.

\begin{prop}
Let ${\bf{v}}=(r,\nu,\Delta)$ be as above, and $\mtc{E}\in\mtc{P}_H({\bf{v}})$ be a general sheaf. If $\mtc{E}$ satisfies the weak DL-condition, then $\pi_{*}\mtc{E}$ is $\mu_H$-stable.
\end{prop}

\begin{proof}
Since $\mtc{E}$ is general, then one has $\mtc{E}|_{E_i}\simeq\mtc{O}_{E_i}^{r-d_i}\oplus \mtc{O}_{E_i}(-1)^{d_i}$ for every $i=1,...,m$. In particular, $\pi_{*}\mtc{E}$ is a vector bundle on $\mb{P}^2$ which fits in the short exact sequence $$0\longrightarrow\pi^{*}\pi_{*}\mtc{E}\longrightarrow\mtc{E}\longrightarrow\bigoplus_{i=1}^m\mtc{O}_{E_i}(-1)^{d_i}\longrightarrow0.$$ Applying $\Hom(\cdot,\mtc{E})$ to this sequence, one gets that $$\Ext^1(\mtc{E},\mtc{E})\longrightarrow \Ext^1(\pi^{*}\pi_{*}\mtc{E},\mtc{E})\longrightarrow \bigoplus_{i=1}^m\Ext^2(\mtc{O}_{E_i}(-1)^{d_i},\mtc{E})=0.$$ Using the argument in the proof of Proposition 8.8, one can deduce that $\Ext^1(\pi^{*}\pi_{*}\mtc{E},\mtc{E})\simeq \Ext^1(\pi_{*}\mtc{E},\pi_{*}\mtc{E})$, and hence one may assume that $\pi_{*}\mtc{E}$ is a general sheaf in $\mtc{P}_H({\bf{v}'}(\pi_{*}\mtc{E}))$.

We need to show that $\pi_{*}\mtc{E}$ satisfies the DL-condition on $\mb{P}^2$. Let $\mtc{V}$ be an exceptional bundle on $\mb{P}^2$ of such that $r(\mtc{V})<r$ and $0< \mu_H(\mtc{V})-\mu_H(\pi_{*}\mtc{E})\leq 2$. We claim that $$\chi(\mtc{V},\pi_{*}\mtc{E})\leq 0.$$ Consider the exceptional bundle $\mtc{F}=\pi^{*}\mtc{V}$ on $X$. We have for each $i\geq0$ that $$\Ext^i(\pi^{*}\mtc{V},\mtc{E})\simeq H^i(X,\pi^{*}\mtc{V}^{*}\otimes\mtc{E})\simeq H^i(\mb{P}^2,\mtc{V}^{*}\otimes \pi_{*}\mtc{E})\simeq \Ext^i(\mtc{V},\pi_{*}\mtc{E})$$ since $R^j\pi_{*}\mtc{E}=0$ for $j>0$. In particular, one has $$\chi(\mtc{V},\pi_{*}\mtc{E})=\chi(\pi^{*}\mtc{V},\mtc{E})\leq 0$$ since $\rk(\mtc{F})<r$ and $\mu_A(\mtc{E})\leq \mu_A(\mtc{F})\leq \mu_A(\mtc{E})-(K_X.A)$.

Similarly, for any exceptional bundle $\mtc{V}$ on $\mb{P}^2$ of such that $r(\mtc{V})<r$ and $0< \mu_H(\mtc{E})-\mu_H(\mtc{V})\leq 2$, we have that $$\chi(\pi_{*}\mtc{E},\mtc{V})\leq 0$$ by considering the exceptional bundle $\mtc{F}$ on $X$ given by $$0\longrightarrow\pi^{*}\mtc{V}\longrightarrow\mtc{F}\longrightarrow\bigoplus_{i=1}^m\mtc{O}_{E_i}(-1)^{r(\mtc{V})}\simeq\bigoplus_{i=1}^m\Ext^1(\mtc{O}_{E_i}(-1),\pi^{*}\mtc{V})\otimes\mtc{O}_{E_i}(-1)\longrightarrow0.$$ We can conclude that $\pi_{*}\mtc{E}$ is $H$-stable. Since $\pi_{*}\mtc{E}$ is general, then it is $\mu_H$-stable.

\end{proof}

\begin{theorem}
Let ${\bf{v}}=(r,\nu,\Delta)$ be as above, and $\mtc{E}\in\mtc{P}_H({\bf{v}})$ be a general sheaf. If $\mtc{E}$ satisfies the weak DL-condition, then $\mtc{E}$ is $\mu_A$-stable for $0<\varepsilon_i\ll1$, where $A=H-\sum \varepsilon_iE_i$.
\end{theorem}

\begin{proof}
Recall that for any subsheaf $\mtc{E}'$ of $\mtc{E}$ and any polarization $A=H-\sum \varepsilon_iE_i$, one has $$\mu_A(\mtc{E})-\mu_A(\mtc{E}')\leq 1.$$ Write $\nu(\mtc{E})=\alpha H-\sum\beta_iE_i$ and $\nu(\mtc{E})=\alpha' H-\sum\beta_i'E_i$, then the inequality comes down to saying that $$\alpha'-\alpha-\sum \varepsilon_i(\beta_i'-\beta_i)\leq 1.$$ Setting $\varphi_i\to 1$ and $\varphi_j\to 0$, one obtains that $$\beta_i-\beta_i'\leq \alpha-\alpha'+1.$$ 

Notice that for any subsheaf $\mtc{E}'$ of $\mtc{E}$, since $\pi_{*}\mtc{E}$ is $\mu_H$-stable, then $\alpha'<\alpha$. Thus $\alpha-\alpha'$ has a strictly positive lower bound $\delta$ since $r$ is a finite number. Now let $\sum\varepsilon_i$ satisfy that $$\delta>\frac{\sum\varepsilon_i}{1-\sum\varepsilon_i}.$$ Then we claim that $\mtc{E}$ is $\mu_A$-stable for $A=H-\sum\varepsilon_iE_i$. Indeed, for any subsheaf $\mtc{E}'$ of $\mtc{E}$, one has

\begin{equation}\nonumber
    \begin{split}
        \mu_A(\mtc{E})-\mu_A(\mtc{E}')&=\alpha-\alpha'-\sum_i\varepsilon_i(\beta_i-\beta_i')\\
        &\geq (\alpha-\alpha')-\sum \varepsilon_i(\alpha-\alpha'+1)\\
        &=(1-\sum\varepsilon_i)(\alpha-\alpha')-\sum\varepsilon_i\\
        &\geq (1-\sum\varepsilon_i)\delta-\sum\varepsilon_i>0.
    \end{split}
\end{equation}

\end{proof}

Now let us consider the case when $\mu_H(\pi_{*}\mtc{E})=\mu_H(\mtc{V})$ for some exceptional bundle $\mtc{V}$. Consider the Harder-Narasimhan filtration of $\pi_{*}\mtc{E}$. Either none of the Harder-Narasimhan factors is isomorphic to $\mtc{V}$, or every Harder-Narasimhan factor is isomorphic to $\mtc{V}$. In the former case, one can run the same argument as above to show that $\mtc{E}$ is $\mu_A$-stable for $0<\varepsilon_i\ll1$ if $\mtc{E}$ satisfies the weak DL-condition. 

For the latter case, a general element in $\mtc{P}_H({\bf{v}})$ fits in an exact sequence $$0\longrightarrow\pi^{*}\mtc{V}^{\oplus N}\longrightarrow\mtc{E}\longrightarrow \bigoplus_{i=1}^m\mtc{O}_{E_i}(-1)^{d_i}\longrightarrow0.$$ If $\mtc{E}$ is not $\mu_A$-semistable, then every destabilizing sheaf $\mtc{E}'$ is of the form $$0\longrightarrow\pi^{*}\mtc{V}^{\oplus N'}\longrightarrow\mtc{E}'\longrightarrow \bigoplus_{i=1}^m\mtc{O}_{E_i}(-1)^{d_i'}\longrightarrow0,$$ where $N>N'$, $d_i'\geq d_i$ for every $i$, and for at least one $i$, $d_i'>d_i$. 

Now let ${\bf{v}}={\bf{v}}(\mtc{E})$ be a character with $\mtc{E}$ a general sheaf fitting in an exact sequence $$0\longrightarrow\pi^{*}\mtc{V}^{\oplus N}\longrightarrow\mtc{E}\longrightarrow \bigoplus_{i=1}^m\mtc{O}_{E_i}(-1)^{d_i}\longrightarrow0.$$ It follows from Theorem 8.3 that $\mtc{E}$ is $\mu_A$-stable if and only if there exists an integer $K\geq 2$ and
\begin{enumerate}
    \item $N_1,...,N_k\geq1$ such that $\sum N_i=N$,
    \item $d_i^j\geq0$ for $i=1,...,m$ and $j=1,...,k$ such that $\sum_jd_i^j=d_i$,  $d_i^j\geq d_i^{j+1}$, and for every $j$ at least one strict inequality $d_i^{j}>d_i^{j+1}$ holds,
\end{enumerate}
satisfying the following conditions
\begin{enumerate}[a)]
    \item $M_A({\bf{v}_i})$ is non-empty for any $i$, where ${\bf{v}_i}$ is the character of a sheaf given by an extension $$0\longrightarrow\pi^{*}\mtc{V}^{\oplus N_i}\longrightarrow\mtc{E}_i\longrightarrow \bigoplus_{j=1}^m\mtc{O}_{E_j}(-1)^{d_j^{i}}\longrightarrow0,$$ and
    \item $\chi({\bf{v}_i},{\bf{v}_j})=0$ for any $i<j$.
\end{enumerate}
The condition b) comes down to saying that $$\frac{1}{r_0^2}=\sum_{i=1}^m \frac{d_i^j}{N_jr_0}-\left(\frac{d_i^k}{N_kr_0}-\frac{d_i^j}{N_jr_0}\right)^2$$ for any $k<j$, where $r_0$ is the rank of $\mtc{V}$.

\section{Stable characters on general blow-ups of $\mb{P}^2$}

In this section, we will apply some results in deformation theory to illustrate that the weak Brill-Noether property and the non-emptiness of the moduli space of stable sheaves proven before actually hold on a blow-up of $\mb{P}^2$ along $m$ general points.

Let $X_0$ be a smooth projective surface, and let $\mtc{E}_0$ be a coherent sheaf on $X_0$. Let $X$ be a deformation of $X_0$ over a local Artin ring $C$. By a deformation of $\mtc{E}_0$ over $X$ we mean a coherent sheaf $\mtc{E}$ on $X$, flat over $C$, together with a map $\mtc{E}\longrightarrow\mtc{E}_0$ such that the induced map $$\mtc{E}\otimes_{\mtc{O}_X}\mtc{O}_{X_0}\longrightarrow \mtc{E}_0$$ is an isomorphism.
We know that if $C$ is the ring of dual numbers $D$, and $X=X_0\times_{\mb{C}}D$ is the trivial deformation of $X_0$, then such deformations $\mtc{E}$ always exist, and
they are classified by $\Ext^1_{X_0}(\mtc{E}_0,\mtc{E}_0)$. Now we consider the more general situation over a sequence $$0\longrightarrow J\longrightarrow C'\longrightarrow C\longrightarrow0,$$ where $C$ is a local Artin ring with the residue field $\mb{C}$, $C'$ is another local Artin ring mapping to $C$, and $J$ is an ideal with $\mathfrak{m}_{C'}J=0$ so that $J$ can be considered as a $\mb{C}$-vector space. Suppose we are given $X_0,\mtc{E}_0,X,\mtc{E}$ as above, and further suppose we are given an extension $X'$ of $X$ over $C'$. We ask for an extension $\mtc{E}'$ of $\mtc{E}$ over $C'$, that is, a coherent sheaf $\mtc{E}'$ on $X'$, flat over $C'$, together with a map $\mtc{E}'\longrightarrow\mtc{E}$ inducing
an isomorphism $\mtc{E}'\otimes_{C'}C\longrightarrow\mtc{E}$. We only need to treat the case of a vector bundle since a general sheaf in $\mtc{P}_H({\bf{v}})$ is locally free, in which case $\mtc{E}$ and $\mtc{E}'$ will also be locally free.

\begin{theorem}\textup{(\cite{Har10})}
In the situation as above, assume that $\mtc{E}_0$ is locally free. Then:
\begin{enumerate}[(1)]
    \item If an extension $\mtc{E}'$ of $\mtc{E}$ over $X'$ exists, then $\Aut(\mtc{E}'/\mtc{E})=J\otimes_{\mb{C}}\Hom_{X_0}(\mtc{E}_0,\mtc{E}_0)$.
    \item Given $\mtc{E}$, there is an obstruction in $J\otimes_{\mb{C}}\Ext^2_{X_0}(\mtc{E}_0,\mtc{E}_0)$ to the existence of $\mtc{E}'$.
    \item If an $\mtc{E}'$ exists, then the set of all such is a torsor under the action of $J\otimes_{\mb{C}}\Ext^1_{X_0}(\mtc{E}_0,\mtc{E}_0)$
\end{enumerate}
\end{theorem}

\begin{corollary}
Let $X_0$ be the blow-up of $m$ collinear points and $\mtc{E}_0$ a general member in a prioritary stack $\mtc{P}_H({\bf{v}})$, which is $\mu_A$-stable for some $A=H-\sum \varepsilon_iE_i$. Then
\begin{enumerate}[(1)]
    \item an extension $\mtc{E}'$ always exists, and 
    \item $\Aut(\mtc{E'}/\mtc{E})=J$.
\end{enumerate}
\end{corollary}

Now let $\mtc{H}$ be the Hilbert scheme of $m$ points on $\mb{P}^2$, and $\mtc{U}\subseteq \mtc{H}\times \mb{P}^2$ the universal family. Let $\mtc{X}\rightarrow\mtc{H}$ the blow-up of $\mtc{H}\times\mb{P}^2$ along $\mtc{U}$. This is the family parameterizing blow-ups of $\mb{P}^2$ along $m$ points. Let $h_0\in \mtc{H}$ be the point corresponding to a blow-up of $\mb{P}^2$ along $m$ distinct collinear points. Let $\mtc{H}'$ be the open subscheme of $\mtc{H}$ parameterizing distinct points of $\mb{P}^2$. Then for $0<\varepsilon_i\ll1$, the divisor $A=H-\sum \varepsilon_iE_i$ is ample on the surface $\mtc{X}_h$ for any $h\in\mtc{H}'$. Given a $\mu_A$-stable sheaf $\mtc{E}_{h_0}$ on $\mtc{X}_{h_0}$, by Corollary 9.2, one can always deform $\mtc{E}_{h_0}$ to the nearby surfaces, which is still $\mu_A$-stable. Moreover, if $\mtc{E}_{h_0}$ satisfies the weak Brill-Noether property, then so do its deformations. This gives us the following:

\begin{theorem}\label{6}
Let ${\bf{v}}=(r,\nu,\Delta)$ be a character such that $\nu({\bf{v}})=\alpha H-\sum \beta_iE_i$ with $-1\leq \beta_i\leq 0$ and $\Delta({\bf{v}})\geq0$. Suppose that $\alpha\notin \mathfrak{C}$, where $\mathfrak{C}$ is the set of $H$-slopes on $\mb{P}^2$ of exceptional bundles. If ${\bf{v}}$ satisfies the weak DL-condition, then $M_{\mtc{X}_h,A}({\bf{v}})\neq \emptyset$ for general $0<\varepsilon_i\ll1$ and general $h\in \mtc{H}'$, where $A=H-\sum \varepsilon_iE_i$.
\end{theorem}

\begin{theorem}\label{4}
Let ${\bf{v}}=(r,\nu,\Delta)$ be a character such that $\nu({\bf{v}})=\alpha H-\sum \beta_iE_i$ with $-1\leq \beta_i\leq 0$, $\alpha-\sum\beta_i\geq-1$, and $\Delta({\bf{v}})\geq0$. Let $A=H-\sum \varepsilon_iE_i$ be a polarization on any surfaces $\mtc{X}_h$, $h\in\mtc{H}'$. If $M_{\mtc{X}_{h_0},A}({\bf{v}})$ is non-empty, then ${\bf{v}}$ is a character on $\mtc{X}_h$ satisfying weak Brill-Noether property for general $h\in\mtc{H}'$.
\end{theorem}

\section{Birational geometry of Hilbert schemes of $X_m$}

It is interesting to figure out the birational properties of moduli spaces. In \cite{CH18}, Coskun and Huizenga reveal the relation between the strong Bogomolov inequalities and the nef cone of the moduli space of sheaves on a surface. The ample cone of the Hilbert scheme of points on $\mb{P}^2$ was worked out in \cite{LQZ03} and \cite{ABCH13}. For del Pezzo surfaces, the ample cone of Hilbert schemes of points is described in \cite{BC13} and \cite{BB}. In this section, we will study the ample cone of the Hilbert schemes of points on the blow-up of $\mb{P}^2$ along $m$ distinct collinear points. 

Let $X=X_m$ be the blow-up of $\mb{P}^2$ along $m$ collinear points. Let $X^{[n]}$ denote the Hilbert scheme parameterizing zero-dimensional schemes of length $n$. Let $X^{(n)}=X^n/\mathfrak{S}_n$ denote the n-th symmetric product of $X$. There is a natural morphism $h:X^{[n]}\rightarrow X^{(n)}$, called the \emph{Hilbert-Chow morphism}, that maps a zero dimensional scheme $Z$ of length $n$ to its support weighted by multiplicity. 

In \cite{Fog73}, Fogarty determined the Picard group of $X^{[n]}$ in terms of the Picard group of $X$. A line bundle $\mtc{L}$ on $X$ naturally determines a line bundle $\mtc{L}^{[n]}$ as follows: $\mtc{L}$ gives rise to a line bundle $\mtc{L}\boxtimes\cdots\boxtimes\mtc{L}$ on $X^n$, which is invariant under the action of the symmetric group $\mathfrak{S}_n$, therefore $\mtc{L}\boxtimes\cdots\boxtimes\mtc{L}$ descends to a line bundle $\mtc{L}_{X^{(n)}}$ on $X^{(n)}$; the pull-back $\mtc{L}[n]:=h^{*}\mtc{L}_{X^{(n)}}$ gives the desired line bundle on $X^{[n]}$.

Let $B$ be the class of the exceptional divisor of the Hilbert-Chow morphism, which parameterizes non-reduced schemes in $X^{[n]}$. Fogarty proves in \cite{Fog73} that if the irregularity $q(X)=0$, then $$\Pic(X^{[n]})\simeq \Pic(X)\oplus \mb{Z}\cdot \frac{B}{2}.$$ As a consequence, the N\'{e}ron-Severi space $N^1(X^{[n]})$ is spanned by $N^1(X)$ and $B$.

Let $\mtc{L}$ be an ample line bundle on $X$. Then $\mtc{L}^{\boxtimes n}$ is an ample line bundle on $X^n$ invariant under the $\mathfrak{S}_n$-action. As a result, $\mtc{L}^{\boxtimes n}$ descends to an ample bundle $\mtc{L}^{(n)}$ on $X^{(n)}$. Since the Hilbert-Chow morphism is birational, then the induced line bundle $\mtc{L}[n]$ is big and nef on $X^{[n]}$. However, since $\mtc{L}[n]$ has degree zero on the fibres of the Hilbert-Chow morphism, then it is not ample and hence it lies on the boundary of the nef cone of $X^{[n]}$.

Given a line bundle $\mtc{L}$ on $X$, we consider the short exact sequence $$0\longrightarrow\mtc{I}_Z\otimes \mtc{L}\longrightarrow\mtc{L}\longrightarrow\mtc{L}|_Z\longrightarrow0,$$ which induces an inclusion $H^0(X,\mtc{L}\otimes\mtc{I}_Z)\subseteq H^0(X,\mtc{L})$. Suppose that $N>n$, then the inclusion induces a rational map $$\psi_{\mtc{L}}:X^{[n]}\dashrightarrow \Gr(N-n,N).$$ Denote by $D_L(n):=\psi_{\mtc{L}}^{*}\mtc{O}_{\Gr(N-n,N)}(1)$ the pull-back of $\mtc{O}_{\Gr(N-n,N)}(1)$. By the Grothendieck-Riemann-Roch Theorem, one can show that the class of $D_{\mtc{L}}(n)$ is $$D_{\mtc{L}}(n)=\mtc{L}[n]-\frac{B}{2}.$$ As $\mtc{O}_{\Gr(N-n,N)}(1)$ is very ample, then the base locus of $D_{\mtc{L}}(n)$ is contained in the indeterminacy locus of $\psi_{\mtc{L}}$. If $\psi_{\mtc{L}}$ is a morphism, then $D_{\mtc{L}}(n)$ is base point free and in particular nef.

\begin{defn}
A line bundle $\mtc{L}$ on $X$ is called \textup{$k$-very ample} if the restriction map $$H^0(X,\mtc{L})\longrightarrow H^0(X,\mtc{L}|_Z)$$ is surjective for every zero dimensional scheme $Z$ of length at most $k+1$.
\end{defn}

Let $\mtc{L}$ be an $(n-1)$-very ample line bundle on a surface $X$ and assume that $h^0(X,\mtc{L})=N$ and $h^i(X,\mtc{L})=0$ for $i>0$. Then $H^i(X,\mtc{L}\otimes \mtc{I}_Z)=0$ for any $i>0$ and any $Z\in X^{[n]}$. Let $$\Xi_n\subseteq X^{[n]}\times X$$ be the universal family and let $\pi_1$ and $\pi_2$ be the natural projections. By cohomology and base change, the $\pi_{1*}(\pi_2^{*}\mtc{L}\otimes\mtc{I}_{\Xi_n})$ is a vector bundle of rank $N-n$ on $X^{[n]}$. By the universal property of the Grassmannian, the map $\psi_{\mtc{L}}:X^{[n]}\longrightarrow \Gr(N-n,N)$ is a morphism. It follows that $D_{\mtc{L}}(n)=\mtc{L}[n]-\frac{B}{2}$ is base point free.

\begin{lemma} \textup{(\cite{BS88})}
Let $\mtc{L}_i$ be a $k_i$-ample line bundle on a projective smooth surface $X$, where $i=1,...,n$. Then $\mtc{L}_1\otimes\cdots\otimes\mtc{L}_n$ is $(k_1+\cdots+k_n)$-ample.
\end{lemma}

\begin{proposition}
Let $X$ be the blow-up of $\mb{P}^2$ along $m$ distinct collinear points. Then the divisor $aH+\sum b_i(H-E_i)$ with $a,b_1,...,b_m\geq1$ is very ample. In particular, the divisor $(n-1)(aH+\sum b_i(H-E_i))$ with $a,b_1,...,b_m\geq1$ is $(n-1)$-very ample.
\end{proposition}

\begin{proof}
This is immediate from the criterion that a divisor is very ample if and only if it separates points and tangent directions (\cite{Har77}).
\end{proof}

\begin{theorem}
Let $X$ be the blow-up of $\mb{P}^2$ along $m$ distinct collinear points. Then the nef cone of $X^{[n]}$ is the cone $\alpha H[n]-\sum \beta_iE_i[n]+\gamma\frac{B}{2}$ satisfying the inequalities $$\gamma\leq 0,\quad \beta_i+(n-1)\gamma\geq0,\quad \textup{and }\alpha+(n-1)\gamma\geq \sum\beta_i.$$
\end{theorem}

\begin{proof}
Let $R$ be a general fibre of the Hilbert-Chow morphism over the singular locus of $X^{(n)}$. Then the curve $R$ has the intersection number $(R.B)=-2$ and $(R.\mtc{L}[n])=0$ for any line bundle $\mtc{L}$ on $X$. As a consequence, the coefficient of $B$ in any nef line bundle on $X^{[n]}$ is non-positive.

Let $C$ be a curve in $X$ that admits a $g^1_n$. The morphism $f:C\rightarrow \mb{P}^1$ defined by the $g^1_n$ induces a curve $C(n)$ in $X^{[n]}$.
Now let $L$ be the special line $H-E_1-\cdots-E_m$ on $X$, and then the induced curve $L(n)$ satisfies the intersection numbers $$(L(n).H[n])=1,(L(n).E_i[n])=1,(L(n),B/2)=n-1.$$ It follows that $$\alpha+(n-1)\gamma\geq \sum\beta_i.$$ Similarly, intersecting with $E_i(n)$, one obtains that $$\beta_i+(n-1)\gamma\geq0.$$

On the other hand, let $D=\alpha H[n]-\sum \beta_iE_i[n]+\gamma\frac{B}{2}$ be a divisor satisfying the inequalities. Then we may write $$D=\left(\alpha-\sum \beta_i+(n-1)\gamma\right)H[n]-\gamma\left((n-1)(H+\sum(H-E_i))[n]-\frac{B}{2}\right)$$$$\quad\quad\quad\quad\quad\quad+\sum (\beta_i-(n-1))(H-E_i)[n]$$ as the sum of three nef divisors. Thus we conclude that the nef cone of $X^{[n]}$ is given by the above inequalities.

\end{proof}

\begin{corollary}
The Hilbert scheme $X^{[n]}$ is log Fano.
\end{corollary}

\begin{proof}
Choose a boundary divisor $\Delta=(1-\delta)L[n]+\varepsilon\frac{B}{2}$ such that $0<\delta\ll1$ and $0<(n-1)\varepsilon<\delta$. Then the divisor $$-K_{X^{[n]}}-\Delta=-K_X[n]-\Delta=(2+\delta)H[n]-\sum_i\delta E_i[n]-\varepsilon\frac{B}{2}$$ is ample. As $X^{[n]}$ is smooth, then $(X^{[n]},\Delta)$ is a klt pair, and hence log Fano.
\end{proof}

\begin{Ques}
Let $X$ be the blow-up of $\mb{P}^2$ along $m$ collinear points, and $A=H-\sum \varepsilon_iE_i$ a polarization. Is the moduli space $M_A({\bf{v}})$ a log Fano variety?
\end{Ques}

\bibliographystyle{alpha}
\bibliography{citation}

\newcommand{\etalchar}[1]{$^{#1}$}
\begin{thebibliography}{ABCH13}

\bibitem[ABCH13]{ABCH13}
Daniele Arcara, Aaron Bertram, Izzet Coskun, and Jack Huizenga.
\newblock The minimal model program for the \textup{Hilbert} scheme of points
  on $\mb{P}^2$ and bridgeland stability.
\newblock {\em Advances in mathematics}, 235:580--626, 2013.

\bibitem[BC13]{BC13}
Aaron Bertram and Izzet Coskun.
\newblock \textup{The birational geometry of the Hilbert scheme of points on
  surfaces}.
\newblock In {\em Birational geometry, rational curves, and arithmetic}, pages
  15--55. Springer, 2013.

\bibitem[BHL{\etalchar{+}}16]{BB}
Barbara Bolognese, Jack Huizenga, Yinbang Lin, Eric Riedl, Benjamin Schmidt,
  Matthew Woolf, and Xiaolei Zhao.
\newblock \textup{Nef cones of Hilbert schemes of points on surfaces}.
\newblock {\em Algebra \& Number Theory}, 10(4):907--930, 2016.

\bibitem[BS88]{BS88}
Mauro Beltrametti and Andrew~J Sommese.
\newblock On k-spannedness for projective surfaces.
\newblock In {\em Algebraic geometry}, pages 24--51. Springer, 1988.

\bibitem[CH18a]{CH18}
Izzet Coskun and Jack Huizenga.
\newblock The nef cone of the moduli space of sheaves and strong {B}ogomolov
  inequalities.
\newblock {\em Israel Journal of Mathematics}, 226(1):205--236, 2018.

\bibitem[CH18b]{CH16}
Izzet Coskun and Jack Huizenga.
\newblock Weak {B}rill-{N}oether for rational surfaces.
\newblock {\em Local and Global Methods in Algebraic Geometry, Contemporary
  Mathematics}, 712:81--104, 2018.

\bibitem[CH20]{CH20}
Izzet Coskun and Jack Huizenga.
\newblock {B}rill-{N}oether theorems and globally generated vector bundles on
  {H}irzebruch surfaces.
\newblock {\em Nagoya Mathematical Journal}, 238:1--36, 2020.

\bibitem[CH21]{CH21}
Izzet Coskun and Jack Huizenga.
\newblock Existence of semistable sheaves on {H}irzebruch surfaces.
\newblock {\em Advances in Mathematics}, 381, 2021.

\bibitem[CHW17]{CHW17}
Izzet Coskun, Jack Huizenga, and Matthew Woolf.
\newblock The effective cone of the moduli space of sheaves on the plane.
\newblock {\em Journal of the European Mathematical Society}, 19(5):1421--1467,
  2017.

\bibitem[DLP85]{DLP85}
J-M Dr\'{e}zet and Joseph Le~Potier.
\newblock Fibr{\'e}s stables et fibr{\'e}s exceptionnels sur $\mb{P}^2$.
\newblock In {\em Annales scientifiques de l'{\'E}cole Normale Sup{\'e}rieure},
  volume~18, pages 193--243, 1985.

\bibitem[Fog73]{Fog73}
John Fogarty.
\newblock Algebraic families on an algebraic surface, {II}, the {P}icard scheme
  of the punctual {H}ilbert scheme.
\newblock {\em American Journal of Mathematics}, 95(3):660--687, 1973.

\bibitem[GH94]{GH98}
Lothar G\"{o}ttsche and Andr\'{e} Hirschowitz.
\newblock Weak {B}rill-{N}oether for vector bundles on the projective plane.
\newblock {\em Algebraic geometry \textup{(Catania, 1993/Barcelona, 1994),
  63–74}, Lecture Notes in Pure and Appl. Math., \textup{200} Dekker, New
  York}, 1994.

\bibitem[Har77]{Har77}
Robin Hartshorne.
\newblock {\em Algebraic {G}eometry}, volume~52.
\newblock Springer, 1977.

\bibitem[Har10]{Har10}
Robin Hartshorne.
\newblock {\em Deformation {T}heory}, volume 257.
\newblock Springer, 2010.

\bibitem[HL10]{HL10}
Daniel Huybrechts and Manfred Lehn.
\newblock {\em The geometry of moduli spaces of sheaves}.
\newblock Cambridge University Press, 2010.

\bibitem[Hui16]{Hui16}
Jack Huizenga.
\newblock \textup{Effective divisors on the Hilbert scheme of points in the
  plane and interpolation for stable bundles}.
\newblock {\em J. Algebraic Geom.}, 25:19--75, 2016.

\bibitem[KO95]{KO95}
SA~Kuleshov and Dmitri~Olegovich Orlov.
\newblock Exceptional sheaves on del {P}ezzo surfaces.
\newblock {\em Izvestiya: Mathematics}, 44(3):479, 1995.

\bibitem[Laz04]{Lar04}
Robert~K Lazarsfeld.
\newblock {\em Positivity in algebraic geometry {I}: {C}lassical setting: line
  bundles and linear series}, volume~48.
\newblock Springer, 2004.

\bibitem[LP97]{LeP97}
Joseph Le~Potier.
\newblock {\em Lectures on vector bundles}.
\newblock Cambridge University Press, 1997.

\bibitem[LQZ03]{LQZ03}
W-P Li, Z~Qin, and Q~Zhang.
\newblock \textup{Curves in the Hilbert schemes of points on surfaces}.
\newblock {\em Contemporary Mathematics}, 322:89--96, 2003.

\bibitem[LZ19]{LZ19}
D~Levine and S~Zhang.
\newblock Brill-{N}oether and existence of semistable sheaves for del {P}ezzo
  surfaces, submitted preprint.
\newblock 2019.

\bibitem[Muk84]{Muk87}
Shigeru Mukai.
\newblock \textup{On the moduli space of bundles on K3 surfaces. I}.
\newblock {\em Vector bundles on algebraic varieties (Bombay, 1984)},
  11:341--413, 1984.

\bibitem[Rud94]{Rud94}
Alexei~N Rudakov.
\newblock \textup{A description of Chern classes of semistable sheaves on a
  quadric surface}.
\newblock {\em J. reine angew. Math.}, 453:113–135, 1994.

\bibitem[Wal98]{Wal98}
Charles Walter.
\newblock Irreducibility of moduli spaces of vector bundles on birationally
  ruled surfaces.
\newblock {\em Algebraic geometry \textup{(Catania, 1993/Barcelona, 1994)},
  Lecture Notes in Pure and Appl. Math., \textup{200}}, pages 201--211, 1998.

\end{thebibliography}

\end{document}